%% file: ES-II.tex
\documentclass{amsart}

\usepackage{amsmath}
\usepackage{amsfonts}
\usepackage{epsfig}
\usepackage{graphics}

\usepackage{pgfplots,pgfplotstable}
\usepgfplotslibrary{groupplots}
\usetikzlibrary{shapes,arrows,snakes,calendar,matrix,backgrounds,folding,calc,positioning,spy}
\usepackage{tikz}
\usepackage{booktabs}

\usepackage{siunitx}
\newcommand{\numeoc}[1]{\num[round-precision=1,round-mode=places, scientific-notation=false]{#1}}

\newcommand{\numnum}[1]{\num[round-precision=2,round-mode=places, scientific-notation=false]{#1}}

\usepackage[utf8]{inputenc}
\usepackage[T1]{fontenc}
\usepackage{graphicx}
\usepackage{amssymb,amsmath}
\usepackage{url}
\usepackage{lastpage}

\usepackage{hyperref}
\definecolor{darkblue}{rgb}{0,0,.7}
\hypersetup{colorlinks=true,linkcolor=darkblue,citecolor=darkblue}

\usepackage[scaled=0.9]{helvet}

\usepackage{eqnarray}

\def\norm#1{\|#1\|}

\newcommand{\vertiii}[1]{{\left\vert\kern-0.25ex\left\vert\kern-0.25ex\left\vert #1 
    \right\vert\kern-0.25ex\right\vert\kern-0.25ex\right\vert}}

\def\enorm#1{\|#1\|_{\mathcal C}} 
\def\hnorm#1{\|#1\|_h} 
\def\anorm#1{\|#1\|_{\mathcal A}} 
\def\snorm#1{\mu^{-1/2}\|#1\|_{0}}

   \newcommand\Hdnull{ H_0(\div\!\!:\!\Omega)} 
      \newcommand\Hdg{ H_g(\div\!\!: \!\Omega) } 
      
       \newcommand\Vtwo{[L^2(\Omega)]^\dim } 
  
  \newcommand\Hdiv{ H(\mathrm{div}\!: \!\Omega)}

      \newcommand\Sh{S_h}
   \newcommand\Vh{V_h}

    \newcommand\HH{ [L^2(\Omega)]^{\dim \times \dim} _{{\rm sym}}  } 
   \newcommand\VV{[H_D^1(\Omega)]^\dim}

   \newcommand\bfs{\sigma}
  \newcommand\bft{\tau}

 \newcommand\ph{u_h}
  \newcommand\uoh{ u_h^{a}}
  \newcommand\Ioh{I_h^a}
  
\newcommand\Ch{{\mathcal C}_h}

\newcommand\B{{\mathcal B}}
\newcommand\M{{\mathcal M}}
\renewcommand\H{H({\rm div}\!:\!\Omega)}

\newcommand\comp {\mathcal {C}} 
\newcommand\elas {{\mathcal A}} 
\newcommand\tr{{\rm tr}}

\def\BF0{0}

\newcommand{\eps}{\varepsilon}

\renewcommand{\div}{\mathrm{div}\,}
\newcommand{\jump}[1]{ [\! [#1 ]\!]}
\def\real{\mathbb{R}}

  \def\dim{d} 
  \def\D{D} 
  
  \newtheorem{theorem}{Theorem}
\newtheorem{lemma}{Lemma}

\newtheorem{remark}{Remark}

\numberwithin{equation}{section}

\usepackage{stmaryrd}

\usepackage{todonotes}

\title[Energy norm analysis of Mixed Finite Elements for Elasticity]{Energy norm analysis of exactly symmetric mixed finite elements for linear elasticity}

\author[P.~L.~Lederer]{Philip L. Lederer}
\address{Department of Mathematics and Systems Analysis, Aalto University,
Otakaari 1, Espoo, Finland}
\email{philip.lederer@aalto.fi}

\author[R.~Stenberg]{Rolf Stenberg}
\address{Department of Mathematics and Systems Analysis, Aalto
University, Otakaari 1, Espoo, Finland}
\email{rolf.stenberg@aalto.fi}

\thanks{This work was supported by the Academy of Finland (Decision 324611).}

\begin{document}
 
\maketitle
\begin{abstract} We consider mixed finite element methods for linear elasticity for which the 
symmetry of the stress tensor is exactly satisfied. We derive a   new quasi-optimal a priori
error  estimate uniformly valid with respect to the compressibility.
For the a posteriori error analysis we consider the Prager-Synge
hypercircle principle and introduce a new estimate uniformly valid in
the incompressible limit. All estimates are validated by numerical
examples.
\end{abstract}


\section{Introduction}
The purpose of this paper is to provide an analysis of mixed finite
element methods for linear elasticity. By mixed methods we mean
methods based on the principle of minimum of the complementary energy,
i.e. methods in which the strain energy is expressed with the stress
tensor and the equation of static equilibrium acts as a constraint. The
Lagrange multiplier connected to this constraint is the displacement
vector. By an integration by parts,  this formulation is dual to the
standard formulation of minimizing the total energy where Dirichlet
boundary condition for the displacement now become natural and
traction conditions are enforced. A crucial property of this class of
methods is that each element is point wise in mechanical equilibrium
and along each element edge or face the traction vector is continuous.
These conditions are naturally very appealing,  and that was the
original motivation to develop the methods in the engineering
community \cite{FdV,WH}.  The formulation can be traced by various
energy methods classically used in elastostatics, cf.
\cite{MR1106633}. The first mathematical treatment of this principle
was done  by Friedrichs \cite{Friedrichs1928} (cf. also \cite[Chapter IV.9]{MR0065391}).

Mathematically, the methods (and the corresponding simpler methods for
scalar second order problems) have been analyzed thoroughly, and the
literature is voluminous as can be seen from the monograph by Boffi,
Brezzi and Fortin \cite{MR3097958}. The analysis framework usually
used is the one laid down in the pioneering work by Raviart and Thomas
\cite{RT}. More recently, tools of differential exterior calculus has
been used to gain insight into the methods \cite{MR2269741}, which has
lead to numerous new methods. These analyses uses the $H(\div\!)$ and
$L^2$ norms for the stress and the displacement, respectively. From a
mechanical point of view this can be seen as unnatural since the norms
do not have any physical meaning and hence blur the origin of the
methods, the minimization of the strain energy. In our analysis we
derive the estimates in energy, or closely related, norms. This
approach was initiated in a classical paper by Babu{\v{s}}ka, Osborn
and Pitk\"aranta \cite{BOP}. For the elasticity problem their
technique of mesh dependent norms was first used by Pitk\"aranta and
Stenberg \cite{PS}. The norms used are basically the energy norm for
the stresses and a broken energy norm for the displacement. The 
broken $H^1$ norm dates back, at least, to the  early works on interior penalty methods, cf.  
Arnold \cite{MR664882}, Douglas and Dupont
\cite{MR0440955} and Baker \cite{MR431742}.
During the last decades the interior penalty
methods have received considerable attention, now under the name of
discontinuous Galerkin methods, cf. \cite{DiPietroErn}. 

The
equilibrium property enables the error in the stress to be decoupled
from that of  the displacement (which is of lower order). However, it also
leads to a superconvergence estimate for the difference between the
finite element solution and the $L^2$ projection of the displacement.
Brezzi and Arnold where the first to observe that this can be used in
a local postprocessing yielding a more accurate approximation
\cite{AB}. In Lovadina and Stenberg  \cite{carlorolf} it was shown
that this leads to a simple a posteriori error estimate. 

In this paper we first collect   the techniques and results mentioned
above. In addition, we improve and extend the analysis. For the a
priori analysis we use Gudi's trick for DG methods \cite{Gudi}  in
order to improve the estimate removing the flaw that the exact stress
should be in $H^s$, with $s>1/2$. In addition, we use a second
postprocessing, the so-called Oswald interpolation \cite{DiPietroErn},
to obtain a kinematically admissible displacement, i.e. continuous and
satisfying the Dirichlet boundary conditions. With this we utilize the
classical hypercircle principle of Prager and Synge
\cite{MR25902,MR600655}, which states with a kinematically admissible
displacement and a statically admissible stress approximation, i.e.
one with  exact satisfaction of the  equilibrium and traction boundary
conditions,  it is possible to get an a posteriori estimate in the
form of an equality. This idea we use to derive an a posteriori
estimate that includes oscillation terms in the case when the
equilibrium and the traction boundary are satisfied only
approximately. 

One attractive feature of mixed methods is that they are robust also
in the incompressible limit. For the a priori estimate this is a
consequence of the ellipticity in the kernel in Brezzi's theory of
saddle point problems which is valid. The Prager-Synge hypercircle
estimate, however, breaks down near incompressibility. For this case
we introduce a novel a posteriori estimate.

 The plan of the paper is the following: in the next two sections we
 first recall the elasticity problem and then discuss its mixed
 approximation. We concretize it  by the  Watwood-Hartz (or
 Johnson-Mercier) \cite{WH,JM} method,  and the families of Arnold-Douglas-Gupta \cite{ADG} and Arnold-Awanou-Winther \cite{AAW3D}.
 Then we derive the a priori estimates for both the stress and the
 post processed displacement in Section~\ref{sec::stabandapriori} an
 Section~\ref{sec::postprocessing}, respectively. In
 Section~\ref{sec::hypercircle} we prove both the Prager-Synge and our
 new a posteriori estimate. In the final section numerical results are
 given. 
 
 In a subsequent paper \cite{ls2}, we will consider mixed methods with the symmetry of the stress tensor enforced in a weak sense.

We use the established notation for Sobolev spaces and
finite element methods. We write  $A \lesssim B$ when there exist a positive constant $C$,   that is independent of the mesh
parameter and \emph{in particular} of the two Lam\'e parameters $\mu,
\lambda$ (see below) such that $A \le C B$.  Analogously we define $A \gtrsim B$. 
This means that the
dependency of the two Lam\'e parameters are made explicit in the norms
used.


\section{The Equations of elasticity}

Let $\Omega\subset \real^\dim$ be a polygonal or polyhedral domain.  The
physical unknowns are the displacement vector $u=(u_1,\dots,u_\dim)$ and
the symmetric  stress tensor $\sigma=\{ \sigma_{ij}\}$, $\sigma_{ij}
=\sigma_{ji}$, $i,j=1,\dots, \dim$. The stress tensor is related to the
strain tensor
\begin{equation}
\varepsilon(u) , \quad \varepsilon(u)_{ij} = \frac{1}{2} \Big(  \frac{\partial u_i}{\partial x_j} +\frac{\partial u_j}{\partial x_i}    \Big),
\end{equation}
by a linear constitutive law. In order to be explicit, we consider the plain strain ($\dim=2$) or 3D ($\dim=3$)  problem for an isotropic material. 
 We define  the compliance matrix
\begin{equation}\label{consteq}
\comp \bft = \frac{1}{2\mu} \Big(  \bft -\frac{\lambda}{ 2\mu +\dim\lambda} \tr(\bft)I \Big),
\end{equation}
with the Lam\'e parameters
 $\mu \mbox{ and } \lambda$. 
 Then it holds
 \begin{equation}
 \comp \bfs +\varepsilon(u)=0.
 \end{equation}
The inverse of the compliance matrix,  the elasticity matrix, we denote by $\elas$,
 i.e. 
 \begin{equation} \label{eq::inccompl}
 \elas \tau = \comp^{-1}\tau =2\mu \tau +\lambda \tr(\tau) I. 
 \end{equation} 
 In the limit   $\lambda\to \infty$ the material becomes incompressible, i.e. 
 \begin{equation}
 \div u=0.
 \end{equation}
The loading consists of a body load $f$ and a traction $g$ on the boundary part $\Gamma_N$. On the complementary part $\Gamma_D$ homogeneous Dirichlet conditions for the displacement are given.
 The 
 equations of elasticity in mixed form are then
\begin{eqnarray}
\comp \bfs -\varepsilon( u )&=&0\  \mbox{ in  }  \Omega,\\
\div \bfs + f&=&0 \  \mbox{ in  }  \Omega, \label{e-eqs}
\\
u&=&0 \  \mbox{ on } \Gamma_D,
\\
\bfs n&=&g \ \mbox{ on } \Gamma_N. \label{neumann} \end{eqnarray} For
this system we use  two variational formulations. The first is: find
$\bfs \in \HH$ and $u\in \VV=\{\,  v\in [H^1(\Omega)]^\dim\, \vert \,
v\vert_{\Gamma_D}=0 \,\}$ such that 
\begin{equation}
\B (\bfs, u; \bft, v) =(f,v) + \langle   g, v \rangle_{\Gamma_N} \quad \forall(\bft,v) \in \HH \times \VV,
\end{equation}
with the bilinear form 
\begin{equation}
\B (\bfs, u; \bft, v) =  (\comp \bfs,\bft ) - ( \varepsilon(u),\bft) - (\varepsilon(v) , \bfs). 
\end{equation}
Physically, the natural norms for analyzing this problem are 
 \begin{equation}
\enorm{ \bft } ^2= ( \comp \bfs , \bfs)\   \mbox{ and } \ \anorm{\varepsilon( u )}^2=   (\elas\varepsilon(u)  , \varepsilon(u)),
 \end{equation}
 which are twice the strain energy expressed by the stress and displacement, respectively. The 
 Babu{\v{s}}ka--Brezzi condition is then simply the identity
 \begin{equation}
\sup_{\bft\in \HH } \frac{\big(\bft, \varepsilon(v)\big)}{\enorm{\bft}}= \anorm{ \varepsilon(v)} \quad \forall v\in \VV.
\end{equation}
This gives the following stability estimate (with a known constant, cf. \cite{HSV})
\begin{equation}
\begin{aligned}
 \sup_{
  \eta \in \HH \atop
  z\in \VV} &\frac{\B(\bft,v; \eta, z)} {\big(\enorm{\eta}^2 + \anorm{ \varepsilon(z)}^2\big)^{1/2} } \geq \Big(\frac{\sqrt{5}-1}{2} \Big)
 \Big(    \enorm{\bft}^2+  \anorm{ \varepsilon(v)}^2\Big)^{1/2}
 \\ & \qquad \qquad \forall (\bft,v)\in \HH \times \VV.
 \end{aligned}
 \end{equation}

In the incompressible limit $\enorm{ \cdot } $ does not, however,  define a norm, and the full stability is a consequence of the ellipticity in the kernel.
Instead of using the abstract theory of Brezzi \cite{MR3097958}, we give an explicit proof of stability.
More precisely, for $\bft$ we 
let $\bft^\D$ be the deviatoric part of $\bft$, defined by the condition $\tr(\bft^\D)=0. $  Hence it holds
 \begin{equation}
  \bft = \bft^\D + \frac{1}{\dim} \tr(\bft) I.
  \end{equation}
  By a direct computation we get.
\begin{lemma} It  holds that
  \begin{equation}\label{cbound}
  (\comp  \bft, \bft) = \frac{1}{2\mu} \Vert  \bft^\D\Vert_0^2 + \frac{1} {2\mu +\dim \lambda} \Vert \tr(\bft) \Vert_0^2.
  \end{equation}
  \end{lemma}
   From this it is seen that when  $\lambda \to \infty$,  $\enorm{
   \cdot } $   does not  give control over the pressure part (the
   trace) of the stress tensor. 
   
   To derive estimates valid independently of  $\lambda$ we use the norms 
   \begin{equation}
   \mu^{-1/2} \Vert \bft \Vert_0 \ \mbox{ for }\bft \in  \HH, \ \mbox{ and}   \
     \mu^{1/2} \Vert \varepsilon(v) \Vert_0 \  \mbox{ for }v \in\VV.
   \end{equation}
   For the stability estimate the 
    Babu{\v{s}}ka--Brezzi condition for the Stokes problem \cite{MR851383}
    \begin{equation}\label{stokes}
  \sup_{v\in \VV} \frac{(\div v, q)}{\Vert  \varepsilon (v)\Vert_0} \geq \beta \Vert q \Vert_0 \quad \forall q \in L^2(\Omega),s
  \end{equation}
    is needed. Using this we prove the following stability estimate.
\begin{theorem}\label{incompstab}  It holds that
\begin{equation}
\begin{aligned}
 \sup_{
  \eta \in \HH \atop
  z\in \VV}  \ \frac{\B(\bft,v; \eta, z)} {\big(\mu^{-1} \Vert \eta \Vert_0^2 + \mu \Vert   \varepsilon(z) \Vert_0^2\big)^{1/2} } &\gtrsim
 \big(   \mu^{-1} \Vert \bft \Vert_0^2+ \mu \Vert \varepsilon(v) \Vert_0^2\big)^{1/2}
 \\ &  \forall (\bft,v)\in \HH \times \VV.
 \end{aligned}
 \end{equation}
 \end{theorem}
 \begin{proof} 
Let  $(\bft, v)$ be given. By \eqref{stokes}  there exists $z\in \VV$ such that 
  \begin{equation}\label{uselbb}
  (\div z, \tr(\bft) )\geq \beta  \Vert  \tr(\bft)  \Vert_0^2   \quad \mbox{and }     \Vert \varepsilon(z)\Vert_0=      \Vert \tr(\bft) \Vert_0.
  \end{equation}
  Let $\delta>0 $ and $ \gamma >0$. By the bilinearity we have 
    \begin{equation}
  \begin{aligned}
  \B(&\bft,v; \bft-\gamma \varepsilon(v)  ,-v -\delta z) 
  \\=&\B(\bft,v; \bft ,-v)  
 -\gamma  \B(\bft,v;\varepsilon(v)  ,0) 
-\delta \B(\bft,v; 0 ,  z) .
  \end{aligned}
  \end{equation}
For the first term we have
  \begin{equation}
  \B(\bft,v; \bft ,-v)  = (\comp \bft, \bft) = \frac{1}{2\mu} \Vert  \bft^\D\Vert_0^2 + \frac{1} {2\mu +\dim \lambda} \Vert \tr(\bft) \Vert_0^2.
  \end{equation}
  From  \eqref{consteq} we have
  \begin{equation}
  \Vert \comp \bft\Vert_0\leq \frac{1}{2\mu} \Big( \Vert \bft \Vert_0 +\frac{1}{\dim} \Vert \tr(\bft ) I \Vert_0 \Big)\leq  \frac{1}{\mu}  \Vert \bft \Vert_0. 
  \end{equation}
  Using the Schwarz and Young inequalities then gives
    \begin{equation}
    \begin{aligned}
     -\gamma  \B(\bft,v;\varepsilon(v)  ,0) &=-\gamma  (\comp \bft, \varepsilon(v) )+\gamma \Vert \varepsilon(v) \Vert_0^2
     \geq -\frac{\gamma}{2} \Vert \comp \bft\Vert_0^2   + \frac{\gamma}{2} \Vert \varepsilon(v) \Vert_0^2  
     \\
     &\geq
      -\frac{\gamma}{2\mu^2} \Vert  \bft\Vert_0^2   + \frac{\gamma}{2} \Vert \varepsilon(v) \Vert_0^2 
      . 
      \end{aligned} \end{equation}
     By \eqref{uselbb} we have
    \begin{equation}
    \begin{aligned}
    -\delta &\B( \bft,v; 0 ,  z)=\delta (\bft , \varepsilon(z)) =   \frac{ \delta}{\dim}   (\tr(\bft) , \div z) + \delta (\bft^\D , \varepsilon(z))
    \\&\geq  \frac{ \delta\beta }{\dim}   \Vert  \tr(\bft)  \Vert_0^2 +\delta (\bft^\D , \varepsilon(z))
    \geq  \frac{ \delta\beta }{\dim}   \Vert  \tr(\bft)  \Vert_0^2 - \delta \Vert \bft^\D \Vert_0 \Vert  \varepsilon(z)\Vert_0
        \\& =  \frac{ \delta\beta }{\dim}   \Vert  \tr(\bft)  \Vert_0^2 - \delta \Vert \bft^\D \Vert_0 \Vert  \tr(\bft)  \Vert_0
          \\& \geq  \frac{ \delta\beta }{2\dim}  \Vert  \tr(\bft)  \Vert_0^2 -\frac{ \delta \dim}{2\beta} \Vert \bft^\D \Vert_0^2.
    \end{aligned}
  \end{equation}
  Collecting the above estimates, we get
  \begin{equation}
   \begin{aligned}
   \B( &\bft,v; \bft-\gamma \varepsilon(v)  ,-v -\delta z) 
   \\   \geq   &\Big( \frac{1}{2\mu} -\frac{ \delta d}{2\beta}\Big)\Vert  \bft^\D\Vert_0^2 +\Big(   
    \frac{1} {2\mu +\dim \lambda}+ \frac{ \delta\beta }{2d}\Big) \Vert \tr(\bft) \Vert_0^2  
 \\
 & -  \frac{\gamma}{2 \mu^2}  \Vert   \bft \Vert_0 ^2 
   +\frac{\gamma}{2}   \Vert \varepsilon(v)\Vert_0^2. 
   \end{aligned}
   \end{equation}
  Now we choose $ \delta=\frac{\beta}{2\mu \dim}  $, which gives 
  \begin{equation}
  \begin{aligned}
   \B( & \bft,v; \bft-\gamma \varepsilon(v) ,-v -\delta z) 
  \\  \geq   &\frac{1}{4\mu} \Vert  \bft^\D\Vert_0^2 +\Big(   
    \frac{1} {2\mu +\dim \lambda}+ \frac{  \beta ^2}{4\dim^2 \mu}\Big) \Vert \tr(\bft) \Vert_0^2  -  \frac{\gamma}{2 \mu^2}  \Vert   \bft \Vert_0 ^2 
   +\frac{\gamma}{2}   \Vert \varepsilon(v)\Vert_0^2. 
    \end{aligned}
  \end{equation}
  Let $C>0$ such that
  \begin{equation}
  \frac{1}{4\mu} \Vert  \bft^\D\Vert_0^2 +\Big(   
    \frac{1} {2\mu +\dim \lambda}+ \frac{  \beta ^2}{4\dim^2 \mu}\Big) \Vert \tr(\bft) \Vert_0^2 \geq \frac{C}{\mu}  \Vert   \bft \Vert_0 ^2, 
  \end{equation}
  and choose $\gamma = C\mu$. This gives 
  \begin{equation}
  B(  \bft,v; \bft-\gamma \varepsilon(v) ,-v -\delta z) \geq \frac{C}{2 } \Big( \mu^{-1} \Vert   \bft \Vert_0 ^2+\mu  \Vert \varepsilon(v)\Vert_0^2\Big).
  \end{equation}
  It also holds
 \begin{equation}
\mu^{-1}\Vert  \bft-\gamma \varepsilon(v)\Vert_0^2 + \mu \Vert \varepsilon(  v +\delta z)\Vert_0^2 \lesssim \mu^{-1}  \Vert \bft \Vert_0^2+ \mu \Vert \varepsilon(v) \Vert_0^2,
 \end{equation}
 which proves the  asserted estimate.
 \end{proof}

The second variational form is the basis for the mixed finite element method. By dualisation the stress is in 
\begin{equation}
\Hdiv= \{ \, \bft\in\HH\, \vert \, \div \bft \in [L^2(\Omega)]^\dim \, \}, 
\end{equation} 
the displacement in $\Vtwo$,  and the bilinear form used is 
\begin{equation}
\M (\bfs, u; \bft, v) =  (\comp \bfs,\bft ) + (  u,\div \bft) +   (v  , \div \bfs),  
\end{equation}
and the formulation is: find $\bfs \in \Hdg$  and $u\in \Vtwo$,  such that
\begin{equation}\label{2ndvariational}
\M (\bfs, u; \bft, v)+(f,v) =0\quad \forall (\bft,v) \in \Hdnull \times \Vtwo,
\end{equation}
with 
\begin{equation}
\Hdg= \{ \, \bft \in \H \, \vert \  \bft n \vert_{\Gamma_N} =g\, \},   \end{equation}
and 
\begin{equation}
 \Hdnull= \{ \, \bft \in \H \, \vert \  \bft n \vert_{\Gamma_N} =0\, \} .
\end{equation}
 
\section{Exactly symmetric mixed finite element methods}
\label{sec::fem} 

The mixed finite element method is based on the variational formulation \eqref{2ndvariational}
using piecewise polynomial 
subspaces $\Vh\subset \Vtwo$ and $\Sh\subset \Hdiv$.
We give a unified presentation that covers the following methods:

 \begin{itemize}
\item The linear triangular method of \emph{Johnson--Mercier} (JM) \cite{JM,WH}. 
\item The triangular family of  \emph{Arnold--Douglas--Gupta} \cite{ADG}.
\item The triangular family of    \emph{Guzman--Neilan}\, \cite{MR3149075}.
\item The tetrahedral family of  \emph{Arnold--Awanou--Winther} \cite{AAW3D}.
\end{itemize}
The triangular or tetrahedral mesh is denoted by $\Ch$. The families are indexed by the polynomial degree $k\geq 2 $ and the displacement space is simply
 \begin{equation} \label{eq::globaldisplspace}
 \Vh= \{\, v\in \Vtwo\, \vert \ v\vert _K\in [P_{k-1}(K)]^\dim\ \forall K\in \Ch\,\}.
 \end{equation}
For the JM method the space in \eqref{eq::globaldisplspace} appears with $k=2$, i.e. discontinuous piecewise linear polynomials are used. 
 
 The spaces for the stress are defined by  
 \begin{equation} \label{eq::globalstressspace}
 \Sh=\{ \, \bft \in \Hdiv\, \vert \, \bft\vert _K\in S(K) \, \forall K \in \Ch\,\}.
 \end{equation}
 The local spaces $S(K)$ are rather involved and here we will not give
 the explicitly definitions. The essential properties are, however, the
 right approximation order which is ensured by the inclusion
 \begin{equation} \label{eq::stressapproxassumption}
  [P_{k}(K)]^{n\times n}_{\mathrm{sym} }\subset S(K),
  \end{equation}
  and the degrees of freedom needed for the stability; 
the local degrees of freedom of $\bft\in S(K) $ 
contain the moments 
\begin{equation}\label{Kmom}
\int_K \bft : \varepsilon(v) \quad \forall v\in [P_{k-1}(K)]^\dim,
\end{equation}
and
\begin{equation}\label{Emom}
\int_E \bft n \cdot v \quad \forall v\in [P_{k}(E)]^\dim, 
\end{equation}
for each edge or face $E$ of $K$.   
For the JM method \eqref{eq::stressapproxassumption}, \eqref{Kmom} and \eqref{Emom} are valid with $k=1$.

For all spaces, except JM, we have
\begin{equation}\label{eqprop}
\div \bft \in \Vh \quad \forall \bft \in \Sh,
\end{equation}
and hence for these it holds
\begin{equation}\label{projpro}
(\div \bft, v-P_hv) = 0 \quad \forall \bft \in \Sh,
\end{equation}
where $P_h :\Vtwo \to \Vh$ denotes the $L^2$-projection. For JM \eqref{eqprop} does not hold. However, it is easily seen that if $\bft\in S(K)$ satisfies 
\begin{equation}
(\div \bft,v)_K = 0 \quad \mbox{for } v\in[P_1(K)]^2
\end{equation}
then $\div \bft=0 $ on $K$. By this  there exist a projection with the property \eqref{projpro}, cf.
\cite[Lemma 2]{JM} and \cite[Lemma 4.3]{PS}.
 
 For an edge/face $E\subset \Gamma_N$ we let $Q_E :  [L^2(E)]^\dim \to
 [P_k(E)]^\dim$ denote the $L^2$ projection and define the operator $Q_h$ by
 $Q_h\vert_E= Q_E$.

The trial and test finite element spaces are then defined as 
\begin{equation}
\begin{aligned}
\Sh^g&=\{ \, \bft\in \Sh\, \vert \, \bft n=Q_h g\mbox{ on } \Gamma_N \,   \}, \\
 \ \Sh^0&=\{ \, \bft\in \Sh\, \vert \, \bft n=0 \mbox{ on } \Gamma_N \,   \}.
 \end{aligned}
\end{equation}
 
The mixed finite element method is:
 find $(\bfs_h,u_h)\in \Sh^g \times \Vh$ such that 
\begin{equation}
\M (\bfs_h, u_h; \bft, v) + (f,v)=0\quad \forall (\bft,v) \in \Sh^0\times \Vh. 
\end{equation}

\section{Stability and a priori error analysis} \label{sec::stabandapriori}
 
 In this section we will derive a priori error estimates. For the
displacement we use the following broken energy norm, first introduced in \cite{PS,RS86}.  Here $\Gamma_h$ denotes the edges/faces in the interior of $\Omega$.  
\begin{equation}
\hnorm{v}^2 = \sum_{K\in \Ch}   \Vert \varepsilon(v) \Vert_{0,K} ^2+ \sum_{E\in \Gamma_{h }}h_E ^{-1} 
\Vert
\jump{v}\Vert_{0,E}^2+ \sum_{E\subset \Gamma_D }h_E ^{-1} 
\Vert
 v\Vert_{0,E}^2\quad  v\in V_h .
\end{equation}

The stability of the method is proven by  the following two
conditions.
\begin{lemma} It holds that
\begin{equation} \label{stability2}
\sup_{\bft\in\Sh^0}  \frac{\big(\div \bft,  v \big)}
{ \norm{\bft}_0}
 \gtrsim  \hnorm{v}  \quad \forall v\in \Vh.
\end{equation}
\end{lemma}
\begin{proof}
 Let $v \in \Vh$ be given.   We choose $\bft \in \Sh^0$ such that all degrees of freedom vanish, except  \eqref{Kmom} and \eqref{Emom} which are chosen such that
 \begin{equation} 
\int_K \bft : \varepsilon(v)= \int_K \vert  \varepsilon(v) \vert^2\quad \forall K \in \Ch,
\end{equation}
\begin{equation} 
\int_E \bft n \cdot v = h_E^{-1} \int_E\vert  \jump{v} \vert^2\quad \forall E \in \Gamma_h,
\end{equation}
 and 
 \begin{equation}
\int_E \bft n \cdot v = h_E^{-1} \int_E\vert  v \vert^2\quad \forall E \subset \Gamma_D .
\end{equation}
Hence
\begin{equation}
(\div \bft, v) = \hnorm{v}^2.
\end{equation}
By scaling it holds
\begin{equation}
\Vert \bft \Vert_0 \lesssim \hnorm{v} ,
\end{equation}
which proves the claim.
\end{proof}
 
\begin{lemma} It holds that
\begin{equation}
\sup_{v\in  V_h} \frac{(v, \div \bft)}{\Vert v \Vert_h } \geq C_1\Vert \tr(\bft) \Vert_0-C_2\Vert \bft^\D \Vert_0 \quad \forall \bft\in \Sh^0.
\end{equation}
\end{lemma}
\begin{proof}
Given $\bft\in  \Sh^0$, \eqref{stokes}  implies that there exists $v
\in [H^1_D(\Omega)]^\dim$ such that
\begin{equation}
(\div v, \tr(\bft)) \geq -\beta \Vert \tr(\bft)\Vert_0^2 \ \mbox{and } \ \Vert \varepsilon(v) \Vert_0=  \Vert \tr(\bft)\Vert_0.
\end{equation}
Let $P_h v\in V_h$ be the projection in \eqref{projpro}.
It holds
\begin{equation}
\begin{aligned}
 (P_h v&, \div \bft) = (v, \div \bft)  
=-(  \varepsilon(v),  \bft) 
\\& =   - \frac{1}{d} (  \div v,  \tr(\bft) )-(  \varepsilon(v),  \bft^\D) 
   \geq\frac{ \beta}{d} \Vert \tr(\bft)\Vert_0^2  -\Vert   \varepsilon(v) \Vert \Vert\bft^\D\Vert_0
  \\& = \Vert \varepsilon(v) \Vert_0\big(  \frac{\beta}{d} \Vert \tr(\bft)\Vert_0    -   \Vert\bft^\D\Vert_0\big).
\end{aligned}
\end{equation}
By scaling we have
\begin{equation}
\hnorm{P_h v} \lesssim \Vert \varepsilon(v) \Vert_0.
\end{equation}
Combining the two estimates above proves the claim.
\end{proof}

In analogy with the proof of Theorem \ref{incompstab}  we then obtain the stability of the mixed method.
\begin{theorem} It holds that
\begin{equation}\label{stab2}
\begin{aligned}
 \sup_{(\eta,z)\in  \Sh^0\times \Vh} \frac{\M(\bft,v; \eta, z)} {\big(\mu^{-1} \Vert \eta \Vert_0^2 + \mu \hnorm{   z }^2\big)^{1/2} } \gtrsim &
 \big(   \mu^{-1} \Vert \bft \Vert_0^2+ \mu \hnorm{v}^2\big)^{1/2}
 \\ & \forall (\bft,v)\in  \Sh^0\times \Vh.
 \end{aligned}
 \end{equation}
 \end{theorem}
 
 We then get the a priori estimate. Here $f_h$ is any piecewise polynomial approximation of $f$.
\begin{theorem} \label{Theo3} It holds that
\begin{equation}
\begin{aligned}
\snorm{&\bfs-\bfs_h} +\mu^{1/2}  \hnorm{ P_h u-\ph }
\\ &\lesssim
\mu^{-1/2} \big(\inf_{\tau \in \Sh^g }\Vert \bfs-\tau \Vert_0 +\big(  \sum_{K\in \Ch} h_K^2 \Vert f- f_h\Vert_{0,K}^2     \big)^{1/2}\big).
  \end{aligned}
  \end{equation}
\end{theorem}
\begin{proof}
  By the stability there exist
  $(\eta, z) \in \Sh^0\times V_h$ with
  \begin{equation}\label{normal}
  \snorm{ \eta }+\mu^{1/2} \hnorm{z} =1,
  \end{equation}
   such that for all
  $\bft \in \Sh^g$ we have
\begin{equation}
\snorm{  \bfs_h-\bft}+\mu^{1/2} \hnorm{P_h u-u_h }\lesssim \M(    \bfs_h-\bft, u_h -P_h u   ; \eta , z).
\end{equation}
By the consistency it holds
\begin{equation}
\M(    \bfs_h-\bft, u_h -P_h u   ; \eta , z)= \M(    \bfs-\bft, u -P_h u   ; \eta , z).
\end{equation}
Writing out gives
\begin{equation}
\M(    \bfs-\bft, u -P_h u   ; \eta , z)=  (\comp( \bfs-\bft) ,  \eta) +  (u -P_h u,\div \eta) + (\div (\bfs-\bft), z) .
\end{equation}
From \eqref{cbound} and \eqref{normal} it follows
\begin{equation}
 (\comp( \bfs-\bft),  \eta)\lesssim \snorm{\bfs-\bft }  \,\snorm{   \eta } \lesssim  \snorm{  \bfs-\bft  } .
\end{equation}
 By the property \eqref{projpro} the second term vanishes
 \begin{equation}
  (u -P_h u,\div \eta) =0.
\end{equation}
Let $\Ioh z $ be the so called Oswald interpolant to $z$, for  
wich it holds \cite{DiPietroErn}
\begin{equation}\label{oswald}
 \Vert   \nabla  \Ioh   z\Vert _0\ + \Big(  \sum_{K\in \Ch}   h_K^{-2} \Vert  z-\Ioh z\Vert _{0,K}^2 \Big)^{1/2} \lesssim \hnorm{z}.
\end{equation}
Using this we obtain
\begin{equation}
(\div(\bfs-\bft), z) = (\div(\bfs-\bft),  z-\Ioh z)  +  (\div(\bfs-\bft), \Ioh z).
\end{equation}
The first term above we treat as follows. First, \eqref{oswald} and \eqref{normal} yield
\begin{equation}
\begin{aligned}
 (\div&(\bfs-\bft),  z-\Ioh z)    =\sum_{K\in \Ch}  (\div(\bfs-\bft),  z-\Ioh z)_K
 \\
 & \leq 
 \sum_{K\in \Ch}  \Vert \div(\bfs-\bft)\Vert_{0,K} \Vert  z-\Ioh z\Vert _{0,K}
 \\ &\leq \Big( \mu^{-1} \sum_{K\in \Ch}  h_K^2 \Vert \div(\bfs-\bft)\Vert_{0,K}^2  \Big) ^{1/2} \Big(  \mu \sum_{K\in \Ch}   h_K^{-2} \Vert  z-\Ioh z\Vert _{0,K}^2 \Big)^{1/2}
 \\
& \lesssim
 \Big( \mu^{-1} \sum_{K\in \Ch}  h_K^2 \Vert \div(\bfs-\bft)\Vert_{0,K}^2  \Big) ^{1/2}.
 \end{aligned}
\end{equation}
By a posteriori error analysis techniques \cite{MR3059294} we have
\begin{equation}
h_K \Vert \div(\bfs-\bft)\Vert_{0,K} \lesssim \big(\Vert \bfs-\bft\Vert_{0,K}+ h_K \Vert f-f_h\Vert_{0,K}\big)  ,
\end{equation}
and hence
\begin{equation}
 (\div(\bfs-\bft), z-\Ioh z) \lesssim
  \mu^{-1/2} \big(\Vert  \bfs-\tau \Vert_0  +\big(  \sum_{K\in \Ch} h_K^2 \Vert f- f_h\Vert_{0,K}^2\big)^{1/2} \big) ,
\end{equation}

Finally, an integration by parts, and \eqref{oswald} and \eqref{normal}, yield
\begin{equation}
\begin{aligned}
   (\div&(\bfs-\bft), \Ioh z)=-    ( \bfs-\bft , \nabla \Ioh z)\leq  \Vert \bfs-\bft\Vert_0 \Vert \nabla \Ioh z\Vert _0
   \\
   & \lesssim 
   \Vert \bfs-\bft\Vert_0 \Vert  z\Vert _h
     \lesssim \mu^{1/2}  \Vert \bfs-\bft\Vert_0.
   \end{aligned}
\end{equation} 
 Collecting  the estimates proves the claim.
\end{proof}

The above result shows that for an exact, sufficiently smooth, solution we
get (by standard arguments) the expected convergence result
\begin{align} \label{eq::sigmasmootherror}
  \| \sigma - \sigma_h\|_0 = \mathcal{O}(h^{k+1}).
\end{align}
The estimate for $ \hnorm{ P_h u-\ph }$ is a superconvergence result,
i.e. we have (again for a sufficiently smooth exact solution)
\begin{equation}\label{crude}
   \hnorm{ P_h u-\ph }= {\mathcal O}(h^{k+1}) \ \mbox{ wheras } \hnorm{ u-\ph }= {\mathcal O}(h^{k-1} ), \
\end{equation}
  except for JM for which we have  $ \hnorm{ u-\ph }= {\mathcal
  O}(h)$. This property enables the postprocessing of the solution in
  the next section.

\section{Postprocessing of the displacement} \label{sec::postprocessing}

In this section we give a two step postprocessing 
yielding a continuous displacement field with enhanced convergence
 properties. We define
 the following two spaces
\begin{align} 
  V_h^* &=  \{\, v\in \Vtwo\, \vert \ v\vert _K\in [P_{k+1}(K)]^\dim\ \forall K\in \Ch\,\}, \label{eq::globaldisplpostspace}\\
  V_h^a &=V^*_h \cap \VV \label{eq::globaldisplaverspace}.
\end{align}
Further let $P_h^*: L^2(\Omega) \to V_h^*$  denote the $L^2$
projection on $V_h^*$.

{\em Postprocessing. Step I:~} Following  \cite{MR1086845,MR1035181} we obtain a  
discontinuous displacement with an enhanced accuracy: find $u_h^* \in V_h^*$ such that 
\begin{equation}\label{pp}
\begin{aligned}
P_h u_h^ * &= u_h
\\  
(\varepsilon (u_h^*),\varepsilon (v))_K &= (\comp \bfs_h, \varepsilon (v))_K \quad \forall v\in   (I-P_h) V_h^*\vert _K.
\end{aligned}
\end{equation}

\begin{lemma}\label{discest}
It holds that
\begin{equation}
\hnorm{u-u_h^* } \lesssim \Vert u-P_h^* u \Vert_{h}
+ \mu^{-1}\Big(\Vert\bfs-\bfs_h\Vert_{0} +\big(  \sum_{K\in \Ch} h_K^2 \Vert f- f_h\Vert_{0,K}^2\big)^{1/2} \Big).
\end{equation}
\end{lemma}
\begin{proof} By the triangle inequality we have
\begin{equation}
\hnorm{u-u_h^* }\leq \hnorm{P_h^* u-u }+\hnorm{P_h^* u-u_h^* }.
\end{equation}
 Next, we write
\begin{equation}
\begin{aligned}
P_h^* u-u_h^* =P_h^*( u-u_h^*)&=(P_h^*-P_h) (u-u_h^*) + P_h(u-u_h^*)
\\ 
&= (P_h^*-P_h) (u-u_h^*) + (P_h u-u_h).
\end{aligned}
\end{equation}
For convenience we denote
\begin{equation}
v= (P_h^*-P_h) (u-u_h^*)\in (I-P_h) V_h^*\vert _K .
\end{equation}
From \eqref{pp} we get
\begin{equation}
\begin{aligned}
\Vert \varepsilon\big((P_h^*-P_h) &(u-u_h^*)\big)\Vert_{0,K}^2 \\
&=   \big(\varepsilon( (P_h^*-P_h) (u-u_h^*)),    \varepsilon(v)   \big)_{0,K}
\\ 
&= \big(\varepsilon( P_h^* (u-u_h^*),    \varepsilon(v)   \big)_{0,K} - \big(\varepsilon( P_h (u-u_h^*)),    \varepsilon(v)   \big)_{0,K}
\end{aligned}
\end{equation}
Next, using  $P_h^*u_h^*=u_h^*$, $\varepsilon (u) -\comp \bfs=0$, and \eqref{pp} gives 
\begin{equation}
\begin{aligned}
 \big(\varepsilon(  P_h^* (u-u_h^*),    \varepsilon(v)   \big)_{0,K}
  &=  \big(\varepsilon( P_h^* u)-\varepsilon(u_h^*),    \varepsilon(v)   \big)_{0,K}
  \\& =
   \big(\varepsilon( P_h^* u)-\comp \bfs_h,    \varepsilon(v)   \big)_{0,K}
   \\  
  & =\big(\varepsilon( P_h^* u-u),    \varepsilon(v)   \big)_{0,K}+  \big(\comp(\bfs- \bfs_h),    \varepsilon(v)   \big)_{0,K}.
   \end{aligned}
\end{equation}
Combining, we get
\begin{equation}
\Vert (\varepsilon\big(P_h^*-P_h) (u-u_h^*)\big)\Vert_{0,K}\lesssim \big(  \Vert u-P_h^*u\Vert_{0,K} +  \mu^{-1} \Vert \bfs-\bfs_h\Vert_{0,K}  \big) .
\end{equation}
By scaling arguments we have
\begin{equation}
\Vert (P_h^*-P_h) (u-u_h^*)\Vert_{h}\lesssim \Big(\sum_{K\in \Ch} \Vert \varepsilon\big( (P_h^*-P_h) (u-u_h^*)\big) \Vert_{0,K}^2\Big)^{1/2}.
\end{equation}
Combining the estimates gives the claim.
\end{proof}

  {\em Postpostprocessing. Step II: ~} The second step is used to
  derive the final continuous displacement approximation (used for the
  hypercircle technique below) by applying an averaging operator $\Ioh
  :V_h^* \to V_h^a$. Now let  $\uoh = \Ioh u_h^*$, then we have the
  following error estimate.
\begin{theorem} \label{Theo4}  It holds that
\begin{equation}
\begin{aligned}
\mu^{1/2} \norm{\varepsilon(u-&u_h^a) }_0 
  \lesssim \mu^{1/2} \Vert u-P_h^* u \Vert_{h}
\\
&+\mu^{-1/2} \big(\inf_{\tau \in \Sh^g }\Vert \bfs-\tau \Vert_0 +\big(  \sum_{K\in \Ch} h_K^2 \Vert f- f_h\Vert_{0,K}^2     \big)^{1/2}\big).
\end{aligned}
\end{equation}
\end{theorem}
\begin{proof} From the interpolation estimate \cite[Theorem 2.2]{MR2034620} and the discrete Korn inequality  \cite{MR2047078}, it follows that
\begin{equation}
\begin{aligned}
\hnorm{u_h^*-u_h^a}& \lesssim \Big( \sum_{E\in \Gamma_{h }}h_E ^{-1} 
\Vert
\jump{u_h^*}\Vert_{0,E}^2\Big)^{1/2} = \Big( \sum_{E\in \Gamma_{h }}h_E ^{-1} 
\Vert
\jump{u-u_h^*}\Vert_{0,E}^2\Big)^{1/2}
\\& \leq \hnorm{u-u_h^*}   .
\end{aligned}
\end{equation}
By the triangle inequality we have
\begin{equation}
\norm{\varepsilon(u-u_h^a) }_0 \leq \hnorm{u-u_h^*}  +\hnorm{u_h^*-u_h^a} 
\end{equation}
and the claim follows from Lemma \ref{discest}.
\end{proof}

For a sufficiently smooth solution we now have 
\begin{equation} \label{eq::conv_aver}
\norm{\varepsilon(u-u_h^a) }_0= {\mathcal O}(h^{k+1} ), 
\end{equation}
 which should be compared to \eqref{crude}. 
 
 \section{A posteriori error estimates by the hypercircle Theorem} \label{sec::hypercircle}
  
  First we recall the hypercircle method \cite{ MR25902} and include its proof.
\begin{theorem} {\rm (The Prager-Synge hypercircle theorem)}   Suppose that: 

\begin{itemize}
\item The stress $\Sigma\in \Hdiv$ is 
      statically admissible;  $\div \Sigma +f=0 \mbox{ in } \Omega $, and $\Sigma n=g \mbox{ on } \Gamma_N$.
      
      \item The displacement $U\in [H^1(\Omega)]^\dim$ is kinematically admissible; $U\vert_{\Gamma_D} =0. $
\end{itemize}
  Then it holds
      \begin{equation}
\begin{aligned}
    \enorm{\bfs -\Sigma} ^2+ \enorm{\bfs -\elas \varepsilon(U)}^2   =   \enorm{\Sigma-\elas \varepsilon(U)}^2 
    \end{aligned}\end{equation}
   and
     \begin{equation}
    \begin{aligned}
    \enorm{\bfs -&\frac{1}{2} \big(\Sigma+ \elas \varepsilon(U)\big)}       = \frac{1}{2}  \enorm{\Sigma-\elas \varepsilon(U)}.
  \end{aligned}
      \end{equation}
      \end{theorem}
  \begin{proof}
We  have
\begin{equation}
\begin{aligned}
 \enorm{ \Sigma-\elas \varepsilon(U)}^2 &= \enorm{ \Sigma-\bfs + \bfs-\elas \varepsilon(U)}^2 
 \\ &= \enorm{ \Sigma-\bfs  }^2  + \enorm{  \bfs-\elas \varepsilon(U)}^2  +  ( \comp(  \Sigma-\bfs) , \bfs-\elas \varepsilon(U) ) . 
 \end{aligned}
\end{equation}
Next,  the symmetry of $\comp$, $ \elas=\comp^{-1} $, and $\varepsilon(u) =\comp \bfs$, give
\begin{equation}
\begin{aligned}
( \comp(  \Sigma-\bfs) , \bfs-\elas \varepsilon(U) ) &=(   \Sigma-\bfs , \comp(\bfs-\elas \varepsilon(U) ))
\\& = (   \Sigma-\bfs ,  \varepsilon(u)- \varepsilon(U) ) .
\end{aligned}
\end{equation}
An integration by parts yields
\begin{equation}\label{orthogonality}
\begin{aligned}
(   \Sigma-\bfs ,  \varepsilon(u)- \varepsilon(U) )) &= -(  \div( \Sigma-\bfs) ,  u-U) ) +\langle    ( \Sigma-\bfs)n, u-U\rangle_{\Gamma_N} 
\\&=  -(  \div \Sigma+f ,  u-U) ) +\langle    ( \Sigma n-g, u-U\rangle_{\Gamma_N}=0,
\end{aligned}
\end{equation}
which proves the first identity. 

Next, the orthogonality \eqref{orthogonality} also yields 
\begin{equation}
\begin{aligned}
\enorm{  \bfs -\frac{1}{2} \big( \Sigma +&\elas \varepsilon(U)  \big) } ^2 \\
&=\enorm{  \elas \varepsilon(u) - \frac{1}{2} \big( \Sigma +\elas \varepsilon(U)  \big) } ^2
\\&= \frac{1}{4 }\enorm{ 2(  \elas \varepsilon(u)  -\elas \varepsilon(U))) +(\elas \varepsilon(U))-\Sigma}^2
\\&
=  \enorm{   \elas \varepsilon(u)  -\elas \varepsilon(U)}^2 + ( \comp(    \elas \varepsilon(u)  -\elas \varepsilon(U) ),   \elas \varepsilon(U)-\Sigma) 
\\
&\qquad +\frac{1}{4 }\enorm{ \elas \varepsilon(U)-\Sigma}^2
\\&
 =( \comp(    \elas \varepsilon(u)  -\elas \varepsilon(U) ),   \elas \varepsilon(u)-\Sigma) 
+\frac{1}{4 }\enorm{ \elas \varepsilon(U)-\Sigma}^2
\\&
=(  \varepsilon(u)  -\varepsilon(U) ,   \elas \varepsilon(u)-\Sigma) 
+\frac{1}{4 }\enorm{ \elas \varepsilon(U)-\Sigma}^2
\\&
= \frac{1}{4 }\enorm{ \elas \varepsilon(U)-\Sigma}^2.
\end{aligned}
\end{equation}
\hfill    \end{proof}
   
  Let $P_K=P_h\vert_K$, where $P_h$ is the projection \eqref{projpro}. 
  The a posteriori estimate is now.
\begin{theorem} \label{th::estimator}
It holds that
\begin{equation}
\begin{aligned}
    \enorm{\bfs -\bfs_h} ^2+ \enorm{\bfs -\elas \varepsilon(\uoh)}^2   \lesssim  \enorm{\bfs_h-\elas \varepsilon(\uoh)}^2 +osc(f)^2 + osc(g)^2 .
    \end{aligned}\end{equation}
   and
     \begin{equation}\label{miracle}
    \begin{aligned}
    \enorm{\bfs -&\frac{1}{2} \big(\bfs_h+ \elas \varepsilon(\uoh)\big)}       \leq\frac{1}{2}  \enorm{\bfs_h-\elas \varepsilon(\uoh)}
    +osc(f)+osc(g)
  \end{aligned}
    \end{equation}
    with 
  \begin{equation}\label{oscf}  osc(f)=C \big(  \sum_{K\in \Ch} h_K^2 \Vert f-P_K f \Vert_{0,K}^2     \big)^{1/2}
  \end{equation}
  and 
  \begin{equation}\label{oscg}
  osc(g)=C\big(  \sum_{E\subset \Gamma_N} h_E \Vert  g-Q_E g   \Vert_{0,E}^2\big)^{1/2}.
  \end{equation}
  \end{theorem}
\begin{proof} Let $(\bar \bfs,\bar u) \in \HH\times \VV$ be the
solution to 
\begin{equation}
\B (\bar \bfs, \bar u; \bft, v) =(P_hf,v) + \langle  Q_h g, v \rangle_{\Gamma_N}  \quad \forall(\bft,v) \in \HH\times \VV.
\end{equation}
Now $(\sigma_h, u_h)$ are admissible approximations to $( \bar \bfs,
 \bar u)$ and the preceding theorem yields 
\begin{equation}
\begin{aligned}
    \enorm{\bar\bfs -\bfs_h }^2+ \enorm{\bar \bfs -\elas \varepsilon(\uoh)}^2   \leq   \enorm{\bfs_h-\elas \varepsilon(\uoh)}^2 
    \end{aligned}
    \end{equation}
   and
     \begin{equation}
    \begin{aligned}
    \enorm{\bar\bfs -&\frac{1}{2} \big(\bfs_h+ \elas \varepsilon(\uoh)\big)}       \leq\frac{1}{2}   \enorm{\bfs_h-\elas \varepsilon(\uoh)}.
  \end{aligned}
      \end{equation}
      For the difference $(\bfs- \bar \bfs, u- \bar u) $ we get
\begin{equation}
\B (\bfs- \bar \bfs, u- \bar u; \bft, v) =(f-P_hf,v) + \langle  g-Q_h g, v \rangle_{\Gamma_N} .
\end{equation}
The stability then yield
\begin{equation}
\enorm{\bfs- \bar \bfs} +\anorm{\eps(u-\bar u)} \lesssim \sup_{\anorm{\eps(v)}=1} \big(     (f-P_hf,v) + \langle  g-Q_h g, v \rangle_{\Gamma_N}     \big).
\end{equation}
By the properties of the projection operators and Korn's inequality we have
\begin{equation}\label{ss}
\begin{aligned}
(f-&P_h f,  v) =(f-P_h f,  v-P_hv)
\\& = \sum_{K\in \Ch}  (f-P_K f, v-P_Kv)_K 
\\&  \leq \sum_{K\in \Ch}  \Vert f-P_K f\Vert_{0,K} \Vert v-P_K v\Vert_{0,K }
\\ &\lesssim\sum_{K\in \Ch} h_K  \Vert f-P_K f\Vert_{0,K} \Vert  \nabla v\Vert_{0,K }
\\&\lesssim \big( \sum_{K\in \Ch} h_K ^2 \Vert f-P_K f\Vert_{0,K} ^2 \big)^{1/2} \Vert \nabla v\Vert_0.
\\&\lesssim 
  \big( \sum_{K\in \Ch} h_K ^2 \Vert f-P_K f\Vert_{0,K} ^2 \big)^{1/2} \anorm{\eps(v)}. 
\end{aligned}
\end{equation}
 Using the trace theorem, a similar estimate also gives 
\begin{equation}\label{ww}
\langle  (g-Q_h g, v\rangle_{\Gamma_N} \lesssim \big( \sum_{E\subset \Gamma_N} h_E \Vert g-Q_K g \Vert_{0,E} ^2 \big)^{1/2} \anorm{\eps(v)}. 
\end{equation}
The assertion then follows by combining the above estimates. 
\end{proof}

\begin{remark} For $f$ and $g$ smooth, \eqref{oscf}, \eqref{eq::globaldisplspace},  and \eqref{oscg}, \eqref{Emom},  yield
\begin{equation}
osc(f) = \mathcal{O}(h^{k+1}) \mbox{ and } osc(g) = \mathcal{O}(h^{k+3/2}),
\end{equation}
respectively. 
Hence,  only $osc(g)$ is a higher order term. However, in most engineering problems, the loading $f$ is given by the gravity, i.e. it is is a constant, and then $osc(f)$ vanish.

Note also that when the oscillation terms vanish, the estimates become
equalities.
\end{remark}

  \section{An a   posteriori estimator uniformly valid in the incompressible limit}
  
  The drawback of the estimate by  the hypercircle argument is that it
 is formulated in terms of $\enorm{\cdot}$ which, unfortunately,
 ceases to be a norm in the incompressible limit $\lambda\to \infty$.
Then  the stress computed from the displacement, i.e.
   \begin{equation}  \label{compstress} 
 \elas \varepsilon(u_h^a)  =   2\mu \varepsilon(u_h^a)  +\lambda \div u_h^a I ,
 \end{equation} 
 grows without limit unless $\div u_h^a $ will vanish identically. For two space dimensions 
 is well known \cite{MR1174468,MR813691} that in order to have a convergence, uniformly valid in the limit, it is required to use piecewise
 polynomials of degree four or higher. In our knowledge no result of this type is known for the three dimensional problem.

      In this section we will therefore derive the following  a posteriori estimate uniformly valid with respect to the second Lam\'e parameter.
  \begin{theorem} \label{th::estimatorinc}
   It holds 
  \begin{equation}
  \begin{aligned}
   \mu ^{-1/2} &\Vert \bfs-\bfs_h \Vert_0+ \mu^{1/2}\Vert \varepsilon(u-\uoh)\Vert_0   
   \\
   &\lesssim
 \mu^{1/2}  \Vert \comp \bfs_h-\varepsilon(\uoh )  \Vert_0  +osc(f)+osc(g).
  \end{aligned}
  \end{equation}
  \end{theorem}
\begin{proof}
 
  By Theorem \ref{incompstab} there exists  $(\bft,v) \in \HH \times \VV $, with   \hfill \break $\mu^{-1/2}\Vert \bft \Vert_0+ \mu^{1/2} \Vert \varepsilon(v) \Vert_0  =1$, such that
  \begin{equation}
  \begin{aligned}
 (&\Vert \bfs-\bfs_h \Vert_0 +\Vert \varepsilon(u-\uoh)\Vert_0 \lesssim
  \B( \bfs-\bfs_h ,  u-\uoh ; \bft, v)
  \\&
  =(\comp (\bfs-\bfs_h),\bft ) - ( \varepsilon(u-\uoh),\bft) - (\varepsilon(v) , \bfs-\bfs_h).
  \end{aligned} 
\end{equation}
Since $ \comp \bfs -\varepsilon(u)=0 $ 
 we have
  \begin{equation}
    (\comp (\bfs-\bfs_h),\bft ) - ( \varepsilon(u-u_h^a),\bft)=   (\comp \bfs_h-\varepsilon(u_h^a),\bft )  
 \leq \Vert  \comp \bfs_h -\varepsilon(u_h^a)\Vert_0\Vert \bft\Vert_0.
  \end{equation}
  Since $\div \bfs_h = P_h f$ and $\sigma_h n=Q_hg$ we get 
  \begin{equation}
  \begin{aligned}
  - (&\varepsilon(v) , \bfs-\bfs_h) = (\div(\bfs-\bfs_h), v  )-\langle  (\bfs-\bfs_h)n, v \rangle_{\Gamma_N}
  \\& = (P_h f -f , v  )-\langle  g-Q_h g, v \rangle_{\Gamma_N}  ,
  \end{aligned}
  \end{equation}
  and  \eqref{ss}, \eqref{ww} give
  \begin{equation}
   - (\varepsilon(v) , \bfs-\bfs_h) \lesssim osc(f)+osc(g).
   \end{equation}
   The assertion follows by collecting the above estimates.
   \end{proof}


\section{Numerical examples}
In this section we validate our theoretical findings with various
numerical examples. All numerical examples were implemented in the
finite element library Netgen/NGSolve, see \cite{netgen}.
For simplicity we only consider the two dimensional case and we use
the following two methods:
\begin{itemize}
  \item The Johnson--Mercier method (JM) from \cite{JM} considers
  linear displacements   and linear stresses. 
    \item The Arnold--Douglas--Gupta method (ADG) from \cite{ADG} where
  we use the  choice of linear displacements and quadratic
  stresses.
\end{itemize}
For both methods we hence have
\begin{equation}
 \Vh= \{\, v\in \Vtwo\, \vert \ v\vert _K\in [P_{1}(K)]^\dim\ \forall K\in \Ch\,\}.
 \end{equation}
As mentioned in Section~\ref{sec::fem} we need to specify the local
stress space $S(K)$ which then defines the global stress space by
\eqref{eq::globalstressspace}. Both, the JM and ADG method use a
similar construction.
Each triangle  $K \in \Ch$ is divided into three sub triangles
$K_i$ with $i = 1,2,3$, by connecting the barycenter with the three
vertices.   $S(K)$ is given
by 
\begin{align*}
  S(K) = \{ \tau \in H(\div, K),  \tau\vert_{K_i}\in [P_{k}(K_i)]^{2 \times 2}, i = 1,2,3. \},
\end{align*} 
with $k=1$ for JM and $k=2$ for ADG. For the postprocessing 
 \begin{equation}
 \Vh^*= \{\, v\in \Vtwo\, \vert \ v\vert _K\in [P_{k+1}(K)]^\dim\ \forall K\in \Ch\,\}
 \end{equation}
is used. Our first example contains a curved boundary and
for that  we use curved elements in order to retrieve the convergence
rates of the analysis. To illustrate this in more details an example
is given in Figure~\ref{fig::curve}. Here we consider an element $K
\in \mathcal{C}_h$ with the vertices $V_0, V_1$ and $V_2$ (the
triangle filled with gray color). Now let  $\Psi_K \in P^{k+2}(K)$
with $ \Psi_K(K) = \tilde K$ be a polynomial mapping from $K$ to the
curved triangle $\tilde K$ (filled with orange color), where we have
chosen the order $k + 2$ as suggested in
\cite{doi:10.1137/15M1045442}. Then, in order to guarantee normal
continuity, see \eqref{eq::globalstressspace}, the stress finite
elements are mapped by a Piola transformation, see \cite{MR3097958},
which includes the mapping $\Psi_K$. Thus if $\widehat \tau$ is a
basis function on a given reference element $\widehat K$ and
$\Phi_{K}$ denotes the linear mapping from $\widehat K$ to $K$, then
the mapped basis function on $\tilde K$ is given by 
\begin{align*}
  \tau = \frac{1}{\operatorname{det}(D( \Psi_K \circ \Phi_{K}))} D( \Psi_K \circ \Phi_{K}) \widehat \tau,
\end{align*}
where $D(\cdot)$ denotes the Jacobian. For more details we refer to
\cite{10.2307/2157925,doi:10.1137/15M1045442}. Note that the mapping
$\Psi_K \circ \Phi_{K}$ is applied for all sub triangles as
illustrated in Figure~\ref{fig::curve}.
 
 
 \begin{figure}[h]
  \centering
  \includegraphics[]{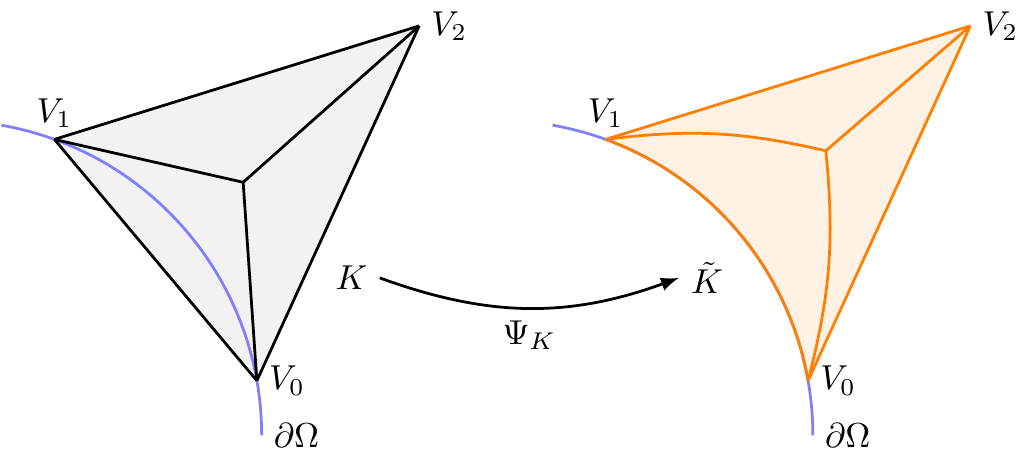}
  \caption{Barycentric refined element at the boundary of the domain
  (left) and corresponding curved element (right). }
  \label{fig::curve}
\end{figure}

We have chosen two classical examples with known exact solutions \cite{szabo1991finite}. In them the material parameters used are 
 the Young modulus $E$ and the Poisson ratio $\nu$. They are related to the 
Lam\'e parameters  by 
\begin{align} \label{eq::lamecoefs}
  \lambda =  \frac{E \nu}{(1+\nu) (1 - 2\nu)}, \quad \textrm{and} \quad \mu = \frac{E}{2 ( 1+ \nu)}.
\end{align}
The incompressible limit is obtained for $\nu\to \frac{1}{2}$. 
The exact solutions are given in polar coordinates  $r$ and $\theta$. 
In the expressions for the exact solutions the parameter value is $\kappa = 3 - 4\nu$.


\subsection{Circular hole in an infinite plate} \label{ex::hole}

The problem is that of the unstressed circular hole with radius $a$ in an infinite plate subject to the unidirectional tension $\sigma_\infty$ as
discussed in Section 19.3.1 in \cite{szabo1991finite}. The exact displacement is 
\begin{align*}
  u_x &= \frac{\sigma_\infty a}{8\mu} \left(\frac{r}{a} (\kappa + 1)\cos(\theta) + 2 \frac{a}{r} ((1+\kappa)\cos(\theta) + \cos(3 \theta)) - 2 \frac{a^3}{r^3} \cos(3 \theta) \right), \\
  u_y &= \frac{\sigma_\infty  a}{8\mu} \left(\frac{r}{a} (\kappa - 3) \sin(\theta) + 2 \frac{a}{r} ((1-\kappa) \sin(\theta) + \sin(3 \theta)) - 2 \frac{a^3}{r^3} \sin(3 \theta) \right),
\end{align*}
and the exact stress components are given by
\begin{equation}\label{holesol}
\begin{aligned} 
  \sigma_{xx} &= \sigma_\infty \left(1 - \frac{a^2}{r^2} (\frac{3}{2} \cos(2 \theta) + \cos(4 \theta)) + \frac{3}{2} \frac{a^4}{r^4} \cos(4 \theta)\right),\\
  \sigma_{yy} &= \sigma_\infty  \left( - \frac{a^2}{r^2}  (\frac{1}{2} \cos(2 \theta) - \cos(4 \theta)) - \frac{3}{2}  \frac{a^4}{r^4}  \cos(4 \theta)\right),\\
  \sigma_{xy} &= \sigma_\infty \left(- \frac{a^2}{r^2}  (\frac{1}{2} \sin(2 \theta) + \sin(4 \theta)) + \frac{3}{2}  \frac{a^4}{r^4}  \sin(4 \theta)\right).
\end{aligned}
\end{equation}

The computations we do for the domain 
 $\Omega = (-b,b) \times
(-w,w) \setminus \Omega_\circ$ with the hole given by $\Omega_\circ =
\{ (x,y) \in \mathbb{R}^2: | (x,y) | \le a\}$, see   the left
picture of Figure~\ref{fig::holeandlshape}. 

\begin{figure}
  \centering
    \includegraphics[width=0.9\textwidth]{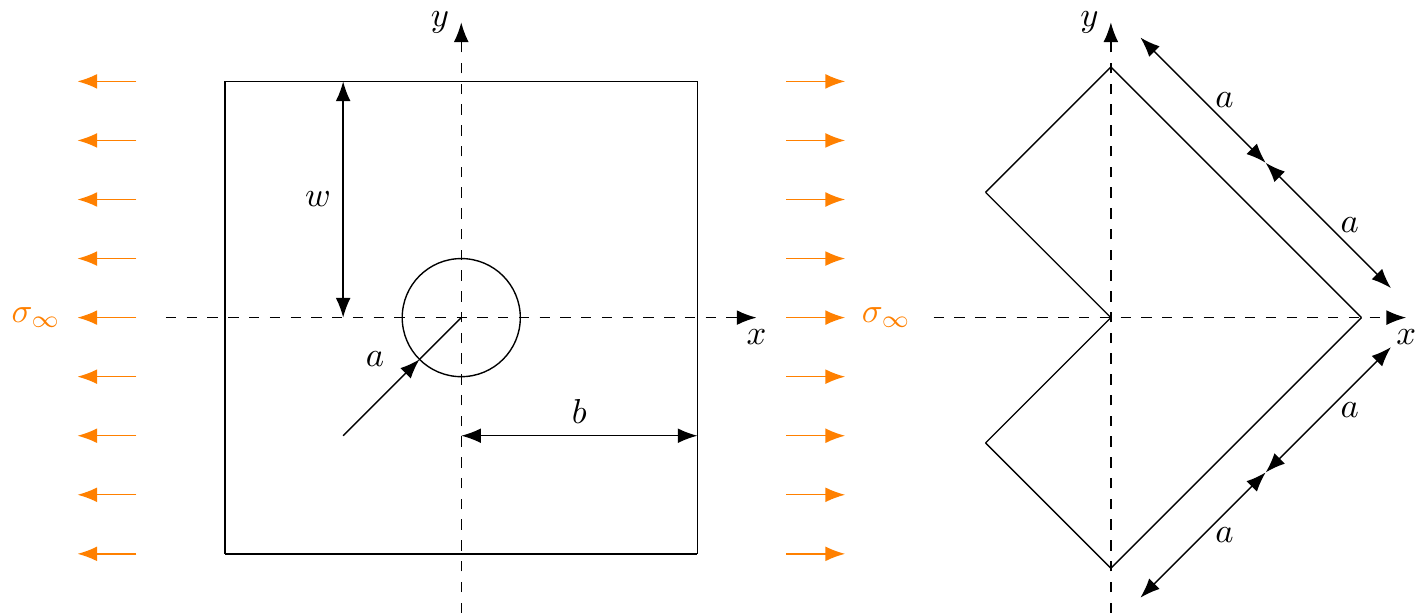}
    \caption{Computational domain for example \ref{ex::hole} on the
    left and for example \ref{ex::lshape} on the
    right.}\label{fig::holeandlshape}
  \end{figure}

We choose
$a = 1$ and $b = w = 4a$, and use the material parameters $E = 1$ and 
$\nu = 0.3,0.4,0.49,0.49999$. 
On the outer boundary we assign the traction obtained  from \eqref{holesol} with  $\sigma_\infty = 1$.
The displacement is fixed to be orthogonal to the rigid body motions.

\begin{figure}
  \centering
  \includegraphics[]{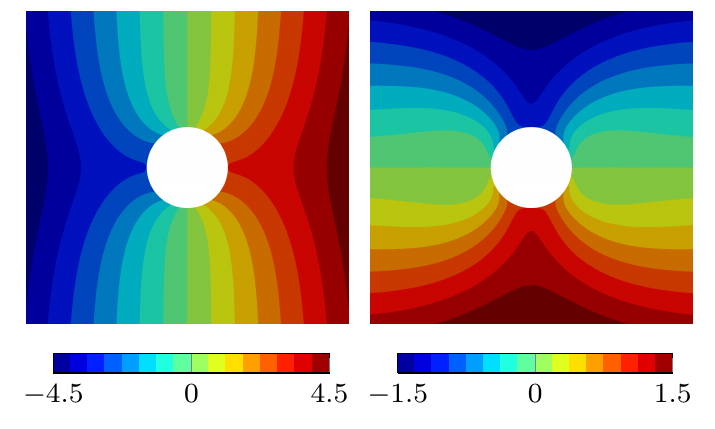}
  \caption{Displacement components $u_x$ and $u_y$ of the solution of example \ref{ex::hole}} \label{fig::solholedispl}
  \end{figure}

\begin{figure}
  \centering
    \includegraphics[]{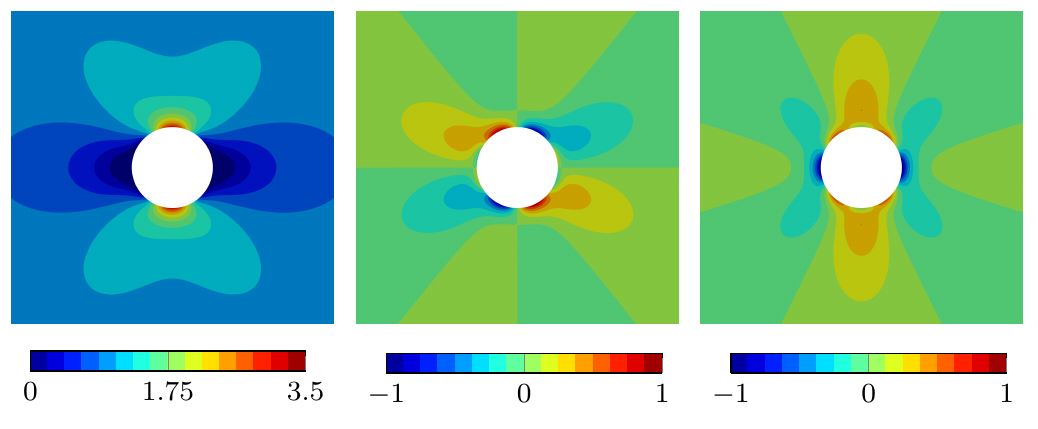}
    \caption{Stress components $\sigma_{xx}$, $\sigma_{xy}$ and
    $\sigma_{yy}$ of the solution of example \ref{ex::hole}}
    \label{fig::solholestress}
  \end{figure}

To validate the theoretical findings   we introduce the  relative
$L^2$-error of the stress and strain
 \begin{equation}
  e_0^{\sigma}  = \frac{ \| \sigma - \sigma_h \|_0 }{\| \sigma\|_0}, 
  \quad \quad
   e_0^{u}  = \frac{\| \eps(u) - \eps(u_h^a) \|_0 }{{\| \eps(u)\|_0}}, 
\end{equation}
for which the a priori estimates of Theorems \ref{Theo3} and \ref{Theo4}, and 
the a posteriori estimate of Theorem
\ref{th::estimatorinc} hold. 

The second set of error quantities are the relative errors in strain energy for the stress directly obtained from the method, and computed from the displacement $u_h^a$, through
\eqref{compstress} \begin{equation}
 e_\mathcal{C}^\sigma  = \frac{\| \sigma - \sigma_h \|_\mathcal{C} }{\| \sigma\|_\mathcal{C}},\quad  \quad
 e_\mathcal{C}^{\mathcal{A}\eps(u)}  = \frac{\| \sigma - \mathcal{A}\eps(u_h^a) \|_\mathcal{C} }{\| \sigma\|_\mathcal{C}}.
\end{equation}
The last set is those given by the hypercircle estimate
\begin{equation}
  {e}^{\operatorname{mean}}_\mathcal{C}   = \frac{\| \sigma -
  \frac{1}{2} (\sigma_h + \mathcal{A}\eps(u_h^a)) \|_\mathcal{C} }{\|
  \sigma\|_\mathcal{C}},\quad  \quad  c_{\operatorname{eff}} =
  \frac{ \| \sigma - \frac{1}{2} (\sigma_h + \mathcal{A}\eps(u_h^a))\Vert_\mathcal{C} }{\frac{1}{2}\|
  \sigma_h -  \mathcal{A}\eps(u_h^a) \|_\mathcal{C}},
\end{equation}
where $c_{\operatorname{eff}}$ measures the efficiency of estimate
\eqref{miracle}. The oscillation is a higher order term, and  
  we expect that $c_{\operatorname{eff}} \rightarrow 1$ when
$h \rightarrow 0$, which means that the error estimator is asymptotically exact. Further, we introduce the symbol $N $ for the number of elements in 
$\mathcal{C}_h$. For uniform refinements we have $h \sim N^{-1/2}$.

In Table~\ref{tab::holeexJM} and Table~\ref{tab::holeexADG} the errors
and the  order of convergence ({{oc}}) for the JM and the
ADG method are given for varying Poisson ratios for a uniform mesh
refinement. As predicted by the analysis all errors converge with
optimal order $\mathcal{O}(N^{-1} )$ and  $\mathcal{O}(N^{-3/2} )$, for JM and ADG, respectively, and the constant
$c_{\operatorname{eff}}$ converges to $1$. Further, the quantities
$e_0^{\sigma}$, $e_0^{u}$ and $e_\mathcal{C}^\sigma$ stay constant for
all values~$\nu$. Since in the incompressible limit $\nu \rightarrow
\frac{1}{2}$,  the Lam\'e parameter $\lambda \rightarrow \infty$, see  
\eqref{eq::lamecoefs}, we expect that the errors, when computed from
\eqref{compstress}, should deteriorate. Indeed, although converging
with optimal order, the errors ${e}^{\operatorname{mean}}_\mathcal{C}$
and $e_\mathcal{C}^{\mathcal{A}\eps(u)}$  start to grow significantly
in the incompressible limit. To give more insight on this behavior we
have plotted in Figure~\ref{fig::blowup} the error
$e_\mathcal{C}^{\mathcal{A}\eps(u)}$ for both methods with respect to
the Poisson ratio. As we can see the blow up occurs continuously and
can be made arbitrarily big if one approaches the incompressible
limit.

We finally note that in all cases the stress $\bfs_h$  is much more
 accurate than that computed by $ \elas \varepsilon(u_h^a) $.
 Furthermore,  the latter dominates in
 ${e}^{\operatorname{mean}}_\mathcal{C} $, thus it holds  
 ${e}^{\operatorname{mean}}_\mathcal{C} \sim\frac{1}{2}  e_\mathcal{C}^{\mathcal{A}\eps(u)}$, for all values of the Poisson ratio.



\begin{table}[]
  \begin{center}
  \footnotesize
        \begin{tabular}{
          r|
          c !{\!(\!\!\!\!}c!{\!\!\!\!)}
          c !{\!(\!\!\!\!}c!{\!\!\!\!)}
          c !{\!(\!\!\!\!}c!{\!\!\!\!)}
          c !{\!(\!\!\!\!}c!{\!\!\!\!)}
          c !{\!(\!\!\!\!}c!{\!\!\!\!)}
          c}
        \toprule
     $N$ & $ e_0^{\sigma} $ & \footnotesize oc & $ e_0^{u} $ &
     \footnotesize oc & $ e_\mathcal{C}^\sigma  $ & \footnotesize oc
     & $ e_\mathcal{C}^{\mathcal{A}\eps(u)}$ & \footnotesize oc & $
     {e}^{\operatorname{mean}}_\mathcal{C}$ & \footnotesize oc &
     $c_{\operatorname{eff}}$ \\ 
    \midrule
    \multicolumn{12}{c}{$\nu = 0.3$} \\
    \input{./conv_data/0.3_circ_JM.tex}\\
    \midrule
    \multicolumn{12}{c}{$\nu = 0.4$} \\
    \input{./conv_data/0.4_circ_JM.tex}\\
    \midrule
    \multicolumn{12}{c}{$\nu = 0.49$} \\
    \input{./conv_data/0.49_circ_JM.tex}\\
    \midrule
    \multicolumn{12}{c}{$\nu = 0.49999$} \\
    \input{./conv_data/0.49999_circ_JM.tex}\\
  \bottomrule
  \end{tabular}
  \vspace{3mm}
  \caption{Errors and  order of convergence for the hole in
  an infinite plate example \ref{ex::hole} for the JM method and
  varying Poisson ratios~$\nu$.} \label{tab::holeexJM}
  \end{center}
  \end{table}

  \begin{table}[]
    \begin{center}
    \footnotesize
          \begin{tabular}{
            r|
            c !{\!(\!\!\!\!}c!{\!\!\!\!)}
            c !{\!(\!\!\!\!}c!{\!\!\!\!)}
            c !{\!(\!\!\!\!}c!{\!\!\!\!)}
            c !{\!(\!\!\!\!}c!{\!\!\!\!)}
            c !{\!(\!\!\!\!}c!{\!\!\!\!)}
            c}
          \toprule
       $N$ & 
       $ e_0^{\sigma} $ & \footnotesize oc &
       $ e_0^{u} $ & \footnotesize oc &
       $ e_\mathcal{C}^\sigma  $ & \footnotesize oc &
       $ e_\mathcal{C}^{\mathcal{A}\eps(u)}$ & \footnotesize oc &
       $ {e}^{\operatorname{mean}}_\mathcal{C}$ & \footnotesize oc & 
       $c_{\operatorname{rel}}$ \\ 
      \midrule
      \multicolumn{12}{c}{$\nu = 0.3$} \\
      \input{./conv_data/0.3_circ_AG2.tex}\\
      \midrule
      \multicolumn{12}{c}{$\nu = 0.4$} \\
      \input{./conv_data/0.4_circ_AG2.tex}\\
      \midrule
      \multicolumn{12}{c}{$\nu = 0.49$} \\
      \input{./conv_data/0.49_circ_AG2.tex}\\
      \midrule
      \multicolumn{12}{c}{$\nu = 0.49999$} \\
      \input{./conv_data/0.49999_circ_AG2.tex}\\
    \bottomrule
    \end{tabular}
    \vspace{3mm}
    \caption{Errors and  order of convergence for the hole in
    an infinite plate Example \ref{ex::hole} for the ADG method and
    varying Poisson ratios~$\nu$.} \label{tab::holeexADG}
    \end{center}
    \end{table}

    \begin{figure}
      \centering
        \includegraphics[width = 0.7\textwidth]{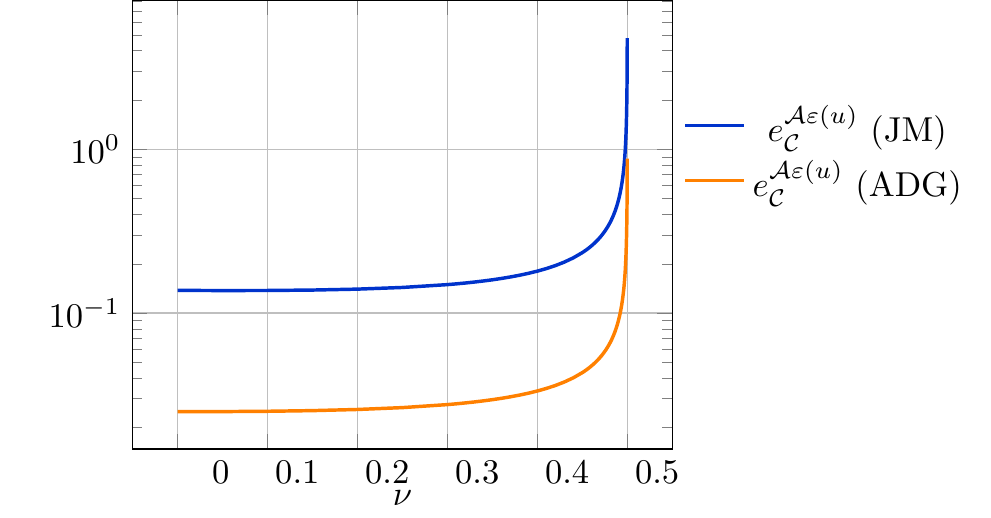}
    \caption{The error $ e_\mathcal{C}^{\mathcal{A}\eps(u)}$ of the JM
    and ADG method for the hole in an infinite plate example
    \ref{ex::hole} for varying Poisson ratios on a fixed mesh with
    $N= 606$}\label{fig::blowup}
    \end{figure}

\subsection{L-shape example} \label{ex::lshape}

We employ an adaptive mesh refinement for the L-shape example from
Section 10.3.2 of \cite{szabo1991finite}. To this end let $\Omega$ be
given by 
\begin{align*}
\Omega = \{ (x,y): |x| + |y| \le {2}^{1/2} a \} \setminus  \{ (x,y): |x- 2^{-1/2} a| + |y| \le 2^{-1/2} a \},
\end{align*}
as illustrated in the right picture of
Figure~\ref{fig::holeandlshape}. For our test case we choose $a = 1$
and use $E = 1$ and different values of $\nu$. Further, we define the
constants \hfill \break $\alpha = 0.544483737 $ and $Q = 0.543075579$. The exact
displacement field, up to rigid-body displacements and rotations, is
given by 
\begin{align*}
  u_x &= \frac{1}{2\mu} r^\alpha  ( (\kappa - Q(\alpha +1))  \cos(\alpha \theta) - \alpha \cos((\alpha - 2)\theta) ), \\
  u_y &= \frac{1}{2\mu} r^\alpha  ( (\kappa + Q(\alpha +1))  \sin(\alpha \theta) + \alpha \sin((\alpha - 2)\theta) ),
\end{align*}
and the stress components are
\begin{align*}
  \sigma_{xx} &= \alpha  r^{\alpha-1}  ( (2 - Q(\alpha + 1))  \cos((\alpha-1)  \theta) - (\alpha-1) \cos((\alpha - 3)\theta) ), \\
  \sigma_{yy} &= \alpha  r^{\alpha-1}  ( (2 + Q(\alpha + 1))  \cos((\alpha-1)  \theta) + (\alpha-1) \cos((\alpha - 3)\theta) ), \\ 
  \sigma_{xy} &= \alpha  r^{\alpha-1}  ( (\alpha-1)  \sin((\alpha -3)\theta) + Q (\alpha+1)  \sin((\alpha-1)\theta)). 
\end{align*}

To solve the problem we again enforce traction boundary conditions  
on the whole boundary $\partial \Omega$.  
We consider a uniform refinement  and 
    adaptive refinements where we use the a
posteriori error estimator of Theorem~\ref{th::estimator} and for the
incompressible limit the estimator given in
Theorem~\ref{th::estimatorinc}. Now let $K \in \Ch$ be an
arbitrary element, then we define the local contributions
\begin{align*}
  \eta(K)^2 =\frac{1}{4} \| \sigma_h - \mathcal{A} \varepsilon(u_h^a)\|_{\mathcal{C},K}^2 \quad \textrm{and} \quad \eta^{\operatorname{inc}}(K)^2 = \mu^{1/2} \| \mathcal{C}\sigma_h - \varepsilon(u_h^a)\|_{0,K}^2,
\end{align*}
where $\| \cdot \|_{\mathcal{C},K}$ reads as the norm $\| \cdot
\|_{\mathcal{C}}$ restricted on the element $K$. The adaptive mesh
refinement loop is defined as usual by 
\begin{align*}
   \mathrm{SOLVE} \rightarrow \mathrm{ESTIMATE} \rightarrow \mathrm{MARK} \rightarrow \mathrm{REFINE} \rightarrow \mathrm{SOLVE} \rightarrow \ldots
\end{align*}
In the marking step we mark an element $K \in \Ch$ for
refinement if $\eta(K) \geq \frac{1}{4} \max\limits_{K \in
\Ch} \eta(K)$ or $\eta^{\operatorname{inc}}(K) \geq
\frac{1}{4} \max\limits_{K \in \Ch}
\eta^{\operatorname{inc}}(K)$. After that, the mesh refinement
algorithm refines the marked elements plus further elements to
guarantee a regular triangulation. Beside the error quantities
introduced above we further define the (relative) estimators
\begin{alignat*}{3}
  \eta &= \frac{\frac{1}{2} \| \sigma_h -  \mathcal{A}\eps(u_h^a) \|_\mathcal{C}}{\|
  \sigma\|_\mathcal{C}},\quad&& \quad &\eta^{\operatorname{inc}} &= \frac{
  \mu^{1/2} \| \mathcal{C}\sigma_h -
  \varepsilon(u_h^a)\|_0}{\mu^{-1/2}\| \sigma\|_0},
\end{alignat*}
and the scaled error 
$$e_0^{u, \operatorname{inc}} =\frac{\mu^{1/2} \|
\eps(u) - \eps(u_h^a) \|_0 }{\mu^{-1/2} \| \sigma\|_0}$$

In Figure~\ref{fig::lshapeerrors} we plot the error history of
$e_{\mathcal{C}}^{\sigma}, \ e_{\mathcal{C}}^{\mathcal{A}\eps(u)}$ and
$e_0^{\operatorname{mean}}$ for the JM and the ADG method using an
adaptive refinement based on the estimator $\eta$ for a moderate
Poisson ratio $\nu = 0.3$. From the coarsest to the finest mesh the measure of efficiency $c_{\operatorname{eff}}$ varies in in the range $0.99 -1.00$, 
and hence the error estimator $\eta$ is not plotted as it would be indistinguishable from $e_0^{\operatorname{mean}}$.
From the figure we see that all 
  errors converge with optimal order
$\mathcal{O}(N^{-(k+1)/2})$. 

To show the drastic decrease of the
errors when using an adaptive algorithm, we also include the error
  $e_{\mathcal{C}}^{\sigma}$ for a uniform refinement.
Since
the exact solution is in the Sobolev space $H^s$, with $s<1.54$, a
uniform mesh only yields a convergence rate of ${\mathcal O}
(h^{0.54})$, i.e. ${\mathcal O}
(N^{-0.27})$.

In Figure~\ref{fig::lshapeerrorsinc} we plot the same quantities but
using an incompressible setting with $\nu = 0.49999$. Although all
quantities still converge with an optimal order, we observe the same
error deterioration as in the previous example. Thus, while
$e_\mathcal{C}^\sigma$ (and also $e_0^{\sigma}$, $e_0^{u}$ which are
not plotted) is not affected by the choice of the Poisson ratio $\nu$,
the errors ${e}^{\operatorname{mean}}_\mathcal{C}$ and
$e_\mathcal{C}^{\mathcal{A}\eps(u)}$ and the estimator $\eta$ are much
bigger and should not be used in practice.

 To this end we follow the theory of
Theorem~\ref{th::estimatorinc}, i.e. we employ the estimator
$\eta^{{\operatorname{inc}}}$. 
In Figure~\ref{fig::lshapeerrorsincinc}
the corresponding relative errors and the estimator are plotted and we
observe that (up to an unknown constant $\mathcal{O}(1)$) the error
estimator gives a good prediction of the errors $e_0^\sigma$ and
$e_0^{u, \operatorname{inc}}$. Further, all errors converge with
optimal error $\mathcal{O}(N^{-(k+1)/2})$, $k=1,2$. Again we include the errors
for a uniform refinement which shows the drastic decrease when
using an adaptive algorithm.

\begin{figure}
  \centering
    \includegraphics[width = \textwidth]{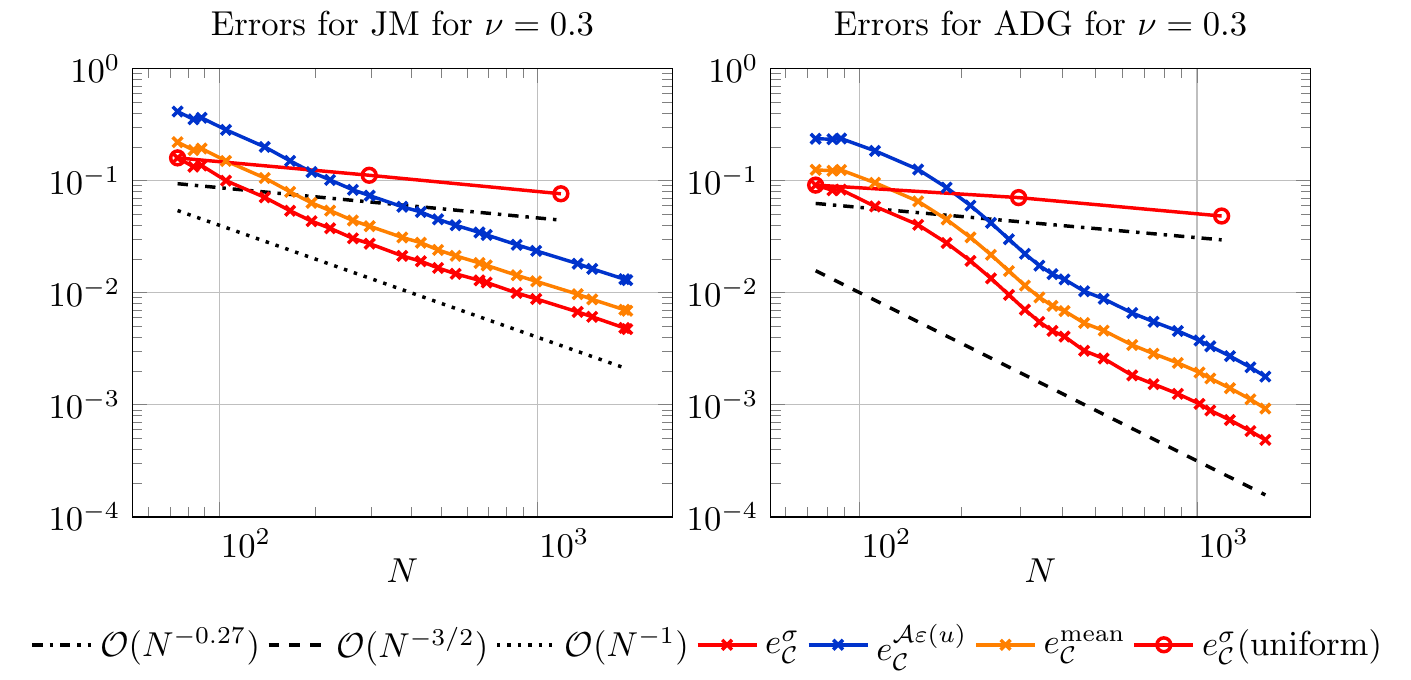}
\caption{Error of the JM and ADG for the L-shape example
\ref{ex::lshape} with an adaptive refinement using estimator $\eta$
and a constant Poisson ratio $\nu = 0.3$.} \label{fig::lshapeerrors}
\end{figure}

\begin{figure}
  \centering
    \includegraphics[width = \textwidth]{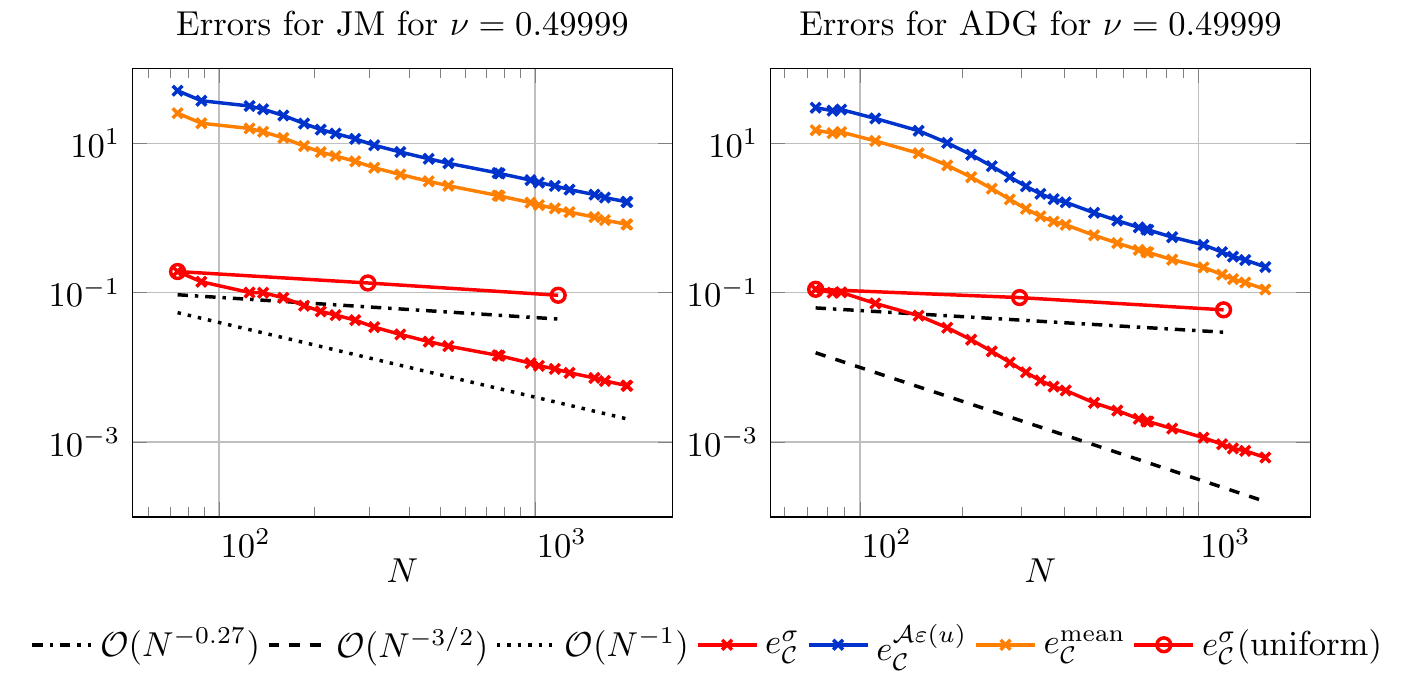}
\caption{Error of the JM and ADG for the L-shape example
\ref{ex::lshape} with an adaptive refinement using estimator $\eta$
and a constant Poisson ratio $\nu = 0.49999$.} \label{fig::lshapeerrorsinc}
\end{figure}

\begin{figure}
  \centering
    \includegraphics[width = \textwidth]{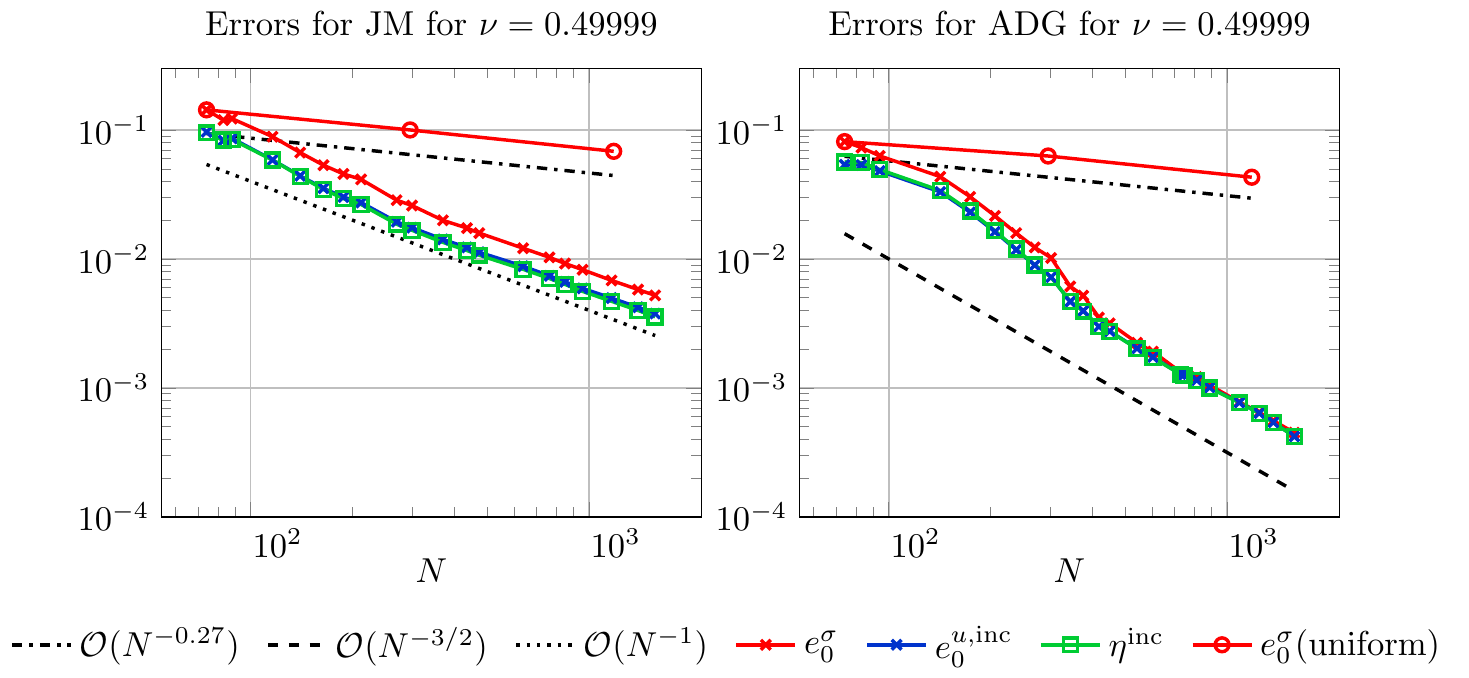}
\caption{Error of the JM and ADG for the L-shape example
\ref{ex::lshape} with an adaptive refinement using estimator
$\eta^{\operatorname{inc}}$ and a constant Poisson ratio $\nu =
0.49999$.}\label{fig::lshapeerrorsincinc}
\end{figure}


\bibliography{literature}
 \bibliographystyle{plain}

  \end{document}

%% file: conv_data/0.3_circ_JM.tex
202&\num{0.0428715966171844}&--&\num{0.10933853546541408}&--&\num{0.04281138063216149}&--&\num{0.1198254448780478}&--&\num{0.06362644980332831}&--&\numnum{0.94176895156639}\\
808&\num{0.018033372592435527}&\numeoc{1.2493529126572407}&\num{0.04531833472764948}&\numeoc{1.270635202786904}&\num{0.018005540554835213}&\numeoc{1.249553449857585}&\num{0.05001435219233172}&\numeoc{1.260520238975607}&\num{0.026587559894286853}&\numeoc{1.2588752468871152}&\numnum{0.9412114044952777}\\
3232&\num{0.00447432548865954}&\numeoc{2.0109271186975257}&\num{0.011833482457088705}&\numeoc{1.9372201440715973}&\num{0.00446669335705982}&\numeoc{2.0111617914633584}&\num{0.013176743868048943}&\numeoc{1.924348245113595}&\num{0.006958719377387909}&\numeoc{1.9338576436962778}&\numnum{0.9473526283591069}\\
12928&\num{0.0010836709878598643}&\numeoc{2.045743402168842}&\num{0.003254160158801607}&\numeoc{1.8625175421745919}&\num{0.001081726232730845}&\numeoc{2.0458717928194186}&\num{0.003626942704962359}&\numeoc{1.8611680468260816}&\num{0.0018928583028228238}&\numeoc{1.8782554127601958}&\numnum{0.9585149960629176}\\

%% file: conv_data/0.4_circ_JM.tex
202&\num{0.042872271648751556}&--&\num{0.1024754621233989}&--&\num{0.04277875982911489}&--&\num{0.14250713937976464}&--&\num{0.07439712675943427}&--&\numnum{0.9578078127206504}\\
808&\num{0.018033450667146195}&\numeoc{1.249369382254343}&\num{0.042844818820757004}&\numeoc{1.2580858395735095}&\num{0.01799022928028807}&\numeoc{1.249681085447182}&\num{0.06014024506719267}&\numeoc{1.2446315458865653}&\num{0.03139392290991741}&\numeoc{1.2447615891460477}&\numnum{0.9582741535902154}\\
3232&\num{0.004474325497171378}&\numeoc{2.0109333620255603}&\num{0.01126705408764215}&\numeoc{1.9270103971461574}&\num{0.0044624727420876946}&\numeoc{2.011298311504084}&\num{0.01610428391114887}&\numeoc{1.90088623515405}&\num{0.008357114977124548}&\numeoc{1.9094084262264213}&\numnum{0.9638657829455269}\\
12928&\num{0.0010836709872449131}&\numeoc{2.0457434057320736}&\num{0.0031224562234713964}&\numeoc{1.851357103488058}&\num{0.0010806506944788402}&\numeoc{2.045943089790939}&\num{0.004478028909150907}&\numeoc{1.8465087645998184}&\num{0.0023036112492355165}&\numeoc{1.8591077123020503}&\numnum{0.9722310441102362}\\

%% file: conv_data/0.49_circ_JM.tex
202&\num{0.04287305941681353}&--&\num{0.09351404260925669}&--&\num{0.042738571987939095}&--&\num{0.3780823242422126}&--&\num{0.19024531089087424}&--&\numnum{0.9936724949866571}\\
808&\num{0.01803356321023105}&\numeoc{1.2493868876608412}&\num{0.03957352844353101}&\numeoc{1.2406473204911632}&\num{0.017971402778475608}&\numeoc{1.249835679972707}&\num{0.16176162533463664}&\numeoc{1.2248310051330504}&\num{0.08138095459074997}&\numeoc{1.225097786756948}&\numnum{0.9939160069843623}\\
3232&\num{0.004474325510196897}&\numeoc{2.0109423613628223}&\num{0.010513160504862775}&\numeoc{1.9123392614265213}&\num{0.004457278086909669}&\numeoc{2.0114681467102953}&\num{0.044458598939629355}&\numeoc{1.8633350088550655}&\num{0.022341227773450486}&\numeoc{1.8649827311504674}&\numnum{0.9950336584191496}\\
12928&\num{0.001083670986448385}&\numeoc{2.045743410992424}&\num{0.0029465985126433487}&\numeoc{1.8350740375285823}&\num{0.0010793268684184106}&\numeoc{2.0460311290977167}&\num{0.012564366691074882}&\numeoc{1.8231245304711392}&\num{0.006305417433904284}&\numeoc{1.8250446833007985}&\numnum{0.9963459122717475}\\

%% file: conv_data/0.49999_circ_JM.tex
202&\num{0.04287315783966082}&--&\num{0.09238403614987771}&--&\num{0.04273322544188468}&--&\num{11.711628690648638}&--&\num{5.855853329120688}&--&\numnum{0.9999933437303044}\\
808&\num{0.018033579228177846}&\numeoc{1.2493889181798972}&\num{0.039159071313996624}&\numeoc{1.2382970264335413}&\num{0.017968902087827816}&\numeoc{1.2498559517684036}&\num{5.019426143745607}&\numeoc{1.2223473801506135}&\num{2.5097292341783124}&\numeoc{1.2223476937592959}&\numnum{0.9999936245013387}\\
3232&\num{0.00447432551211543}&\numeoc{2.010943642188256}&\num{0.010417929465367768}&\numeoc{1.9102779763178426}&\num{0.004456587754570699}&\numeoc{2.011490842526954}&\num{1.3849012603659514}&\numeoc{1.8577393147452697}&\num{0.6924542308790048}&\numeoc{1.8577411036333493}&\numnum{0.999994844575019}\\
12928&\num{0.0010836709863409524}&\numeoc{2.0457434117540583}&\num{0.0029244195706672684}&\numeoc{1.8328463576839544}&\num{0.0010791509353478292}&\numeoc{2.0460428529006447}&\num{0.39229118841003435}&\numeoc{1.8197862844933952}&\num{0.1961463393939172}&\numeoc{1.8197883053678099}&\numnum{0.9999962317803571}\\

%% file: conv_data/0.3_circ_AG2.tex
202&\num{0.009935450421657577}&--&\num{0.024936208216471786}&--&\num{0.009920936170832817}&--&\num{0.027650543651819934}&--&\num{0.014692011673652784}&--&\numnum{0.9414895416653294}\\
808&\num{0.002704842959480119}&\numeoc{1.8770405382191417}&\num{0.011354736722550546}&\numeoc{1.134947851303161}&\num{0.0027002147353413623}&\numeoc{1.8774021217208319}&\num{0.01313189091867395}&\numeoc{1.0742331751346437}&\num{0.00670368261590981}&\numeoc{1.1320061957377638}&\numnum{0.9795609961450003}\\
3232&\num{0.0003361353374963505}&\numeoc{3.0084307105353045}&\num{0.0014344024334322484}&\numeoc{2.9847725093423785}&\num{0.00033553791349952544}&\numeoc{3.0085264471051736}&\num{0.001661827400109419}&\numeoc{2.9822322166598005}&\num{0.0008477086467181453}&\numeoc{2.9833134387168574}&\numnum{0.9802505133313109}\\
12928&\num{4.014037456974219e-05}&\numeoc{3.06591623500267}&\num{0.00016117631364165563}&\numeoc{3.153738193719868}&\num{4.0067213508714365e-05}&\numeoc{3.0659817048180447}&\num{0.00018535673883808134}&\numeoc{3.164394077504409}&\num{9.482092574028021e-05}&\numeoc{3.160291119475505}&\numnum{0.9774457556823412}\\

%% file: conv_data/0.4_circ_AG2.tex
202&\num{0.009935449820564596}&--&\num{0.023737759158593524}&--&\num{0.009912907203510467}&--&\num{0.03348837794204451}&--&\num{0.017465861015168487}&--&\numnum{0.9590646380050155}\\
808&\num{0.002704842957521581}&\numeoc{1.877040451980981}&\num{0.01040414066006606}&\numeoc{1.1900259419622088}&\num{0.002697655218435835}&\numeoc{1.8776022529322478}&\num{0.015711365248640613}&\numeoc{1.0918519487764287}&\num{0.007970971723439794}&\numeoc{1.1317102492800828}&\numnum{0.9856186759217505}\\
3232&\num{0.00033613533753026805}&\numeoc{3.0084307093450953}&\num{0.0013166872817699615}&\numeoc{2.9821731633803346}&\num{0.0003352075142413334}&\numeoc{3.008579577336805}&\num{0.0019925879215771376}&\numeoc{2.9790932612790835}&\num{0.0010103180183389042}&\numeoc{2.979946128334281}&\numnum{0.9861673011919064}\\
12928&\num{4.014130489544717e-05}&\numeoc{3.0658827984715336}&\num{0.00016541372826889798}&\numeoc{2.9927618617843637}&\num{4.0027849143772225e-05}&\numeoc{3.065978488687041}&\num{0.0002459228765016482}&\numeoc{3.0183655333357753}&\num{0.00012458119994813305}&\numeoc{3.0196512038004384}&\numnum{0.9870239688618331}\\

%% file: conv_data/0.49_circ_AG2.tex
202&\num{0.009935450764379442}&--&\num{0.021959584901966978}&--&\num{0.0099030251174504}&--&\num{0.0897714038412075}&--&\num{0.04515948911397405}&--&\numnum{0.9940034917376572}\\
808&\num{0.002704842986813413}&\numeoc{1.8770405734057993}&\num{0.009187894277108912}&\numeoc{1.2570446224204077}&\num{0.0026945049695837335}&\numeoc{1.8778490509025665}&\num{0.04161041932393137}&\numeoc{1.109311128988937}&\num{0.02084892329929099}&\numeoc{1.1150562845887373}&\numnum{0.9979165481861905}\\
3232&\num{0.00033613533818725375}&\numeoc{3.0084307221488333}&\num{0.001165017176210281}&\numeoc{2.979383030928584}&\num{0.00033480084616331683}&\numeoc{3.0086451683307534}&\num{0.0052857539606513595}&\numeoc{2.9767636494273213}&\num{0.002648183569607305}&\numeoc{2.9768978435156104}&\numnum{0.9980039077694394}\\
12928&\num{4.014023401294103e-05}&\numeoc{3.0659212897633172}&\num{0.00013646140287954458}&\numeoc{3.093786367036857}&\num{3.9976743499950515e-05}&\numeoc{3.066070313865968}&\num{0.0006115749850467092}&\numeoc{3.111507969332853}&\num{0.0003064408483560219}&\numeoc{3.111322606423099}&\numnum{0.9978728944755696}\\

%% file: conv_data/0.49999_circ_AG2.tex
202&\num{0.009935451030417826}&--&\num{0.021724837886751825}&--&\num{0.009901711843999371}&--&\num{2.781554647699758}&--&\num{1.3907861852480317}&--&\numnum{0.9999936996171864}\\
808&\num{0.00270484299335733}&\numeoc{1.8770406085460214}&\num{0.009034812938879824}&\numeoc{1.265778775264409}&\num{0.0026940863206158824}&\numeoc{1.8778818885177833}&\num{1.2876487636913612}&\numeoc{1.1111523314103815}&\num{0.6438257955454125}&\numeoc{1.1111583557310094}&\numnum{0.9999978182738153}\\
3232&\num{0.0003361353383258459}&\numeoc{3.008430725044355}&\num{0.0011458733902425554}&\numeoc{2.979047084398344}&\num{0.00033474680133874186}&\numeoc{3.0086539015026923}&\num{0.1635883402748551}&\numeoc{2.976597289043005}&\num{0.08179434172081523}&\numeoc{2.976597430482756}&\numnum{0.999997910511553}\\
12928&\num{4.014024034205332e-05}&\numeoc{3.065921062881201}&\num{0.0001653143261764033}&\numeoc{2.7931639873127803}&\num{3.997012786389718e-05}&\numeoc{3.066076176953594}&\num{0.023689984670187007}&\numeoc{2.787720754031453}&\num{0.011845009215522765}&\numeoc{2.787721724436777}&\numnum{0.9999985784163062}\\